\documentclass{article}

\usepackage{amsthm, mathtools,mathrsfs}
\usepackage{titlesec}
\usepackage{enumitem}
\setenumerate[0]{label=(\roman*)}
\titlelabel{\thetitle.\quad}

\usepackage{amsmath}
\usepackage{amssymb}
\usepackage{relsize}

\usepackage{url}
\usepackage{
    nameref,
    hyperref,
    cleveref
}

\hypersetup{
    unicode=true,
    pdftitle={Doeblin Trees},
    pdfauthor={James T. Murphy III},
    colorlinks,
    linktocpage,
    citecolor=blue,
    linkcolor=blue,
    urlcolor=black,
    filecolor=blue
}
\renewcommand{\leq}{\leqslant}
\renewcommand{\geq}{\geqslant}

\newcommand{\R}{\mathbb{R}}

\newcommand{\F}{\mathcal{F}}

\newcommand{\B}{\mathbf{B}}
\newcommand{\N}{\mathbb{N}}

\renewcommand{\P}{\mathbf{P}}
\newcommand{\E}{\mathbf{E}}
\newcommand{\Z}{\mathbb{Z}}

\newcommand{\TR}{\mathbf{R}}

\newcommand{\Gr}{\Gamma}

\newcommand{\eft}{{EFT}}
\newcommand{\eff}{{EFF}}
\newcommand{\beft}{bridge \eft{}}
\newcommand{\befts}{bridge \efts{}}
\newcommand{\beff}{bridge \eff{}}

\newcommand{\efts}{{EFTs}}
\newcommand{\effs}{{EFFs}}

\newcommand{\deft}{Doeblin \eft{}}
\newcommand{\deff}{Doeblin \eff{}}

\newcommand{\deffs}{Doeblin \effs{}}

\newcommand{\defn}[1]{\textbf{#1}}
\newcommand{\bs}[1]{\boldsymbol{#1}}
\newcommand{\Fam}[1]{{({#1})}}
\newcommand{\set}[1]{\left\{{#1}\right\}}

\renewcommand{\G}{\mathbf{G}}

\newcommand{\hgen}{h_{\mathrm{gen}}}

\newcommand{\fnext}{f_+}
\newcommand{\Xiuniv}{\Xi_{\mathrm{univ}}}

\newcommand{\sq}{\mathsmaller{\square}}
\newcommand{\Psq}{\P^{\mathsmaller{\square}}}
\newcommand{\Esq}{\E^{\mathsmaller{\square}}}

\newtheorem{theorem}{Theorem}[section]

\newtheorem{lemma}[theorem]{Lemma}
\crefname{lemma}{lemma}{lemmas}
\Crefname{lemma}{Lemma}{Lemmas}

\newtheorem{proposition}[theorem]{Proposition}
\crefname{proposition}{proposition}{Propositions}
\Crefname{proposition}{Proposition}{Propositions}

\newtheorem{corollary}[theorem]{Corollary}
\crefname{corollary}{corollary}{corollaries}
\Crefname{corollary}{Corollary}{Corollaries}

\newtheorem{assumption}[theorem]{Assumption}
\crefname{assumption}{assumption}{assumptions}
\Crefname{assumption}{Assumption}{Assumptions}

\newtheorem{remark}[theorem]{Remark}
\crefname{remark}{remark}{remarks}
\Crefname{remark}{Remark}{Remarks}

\theoremstyle{definition}

\crefname{definition}{definition}{definitions}
\Crefname{definition}{Definition}{Definitions}

\newtheorem{example}[theorem]{Example}
\crefname{example}{example}{examples}
\Crefname{example}{Example}{Examples}

\newenvironment{postponedproof}{\noindent\ignorespaces\emph{Proof (sketch).}}{\hfill\\}
\begin{document}

\title{\textsc{Doeblin Trees}}
\author{Fran\c{c}ois~Baccelli\thanks{The University of Texas at Austin, 
\href{mailto:baccelli@math.utexas.edu}{\nolinkurl{baccelli@math.utexas.edu}}},~ 
    Mir-Omid~Haji-Mirsadeghi\thanks{Sharif University of Technology, 
    \href{mailto:mirsadeghi@sharif.edu}{\nolinkurl{mirsadeghi@sharif.edu}}},~ and \vspace{-.4cm}\\
    James~T.~Murphy~III\thanks{The University of Texas at Austin, 
    \href{mailto:james@intfxdx.com}{\nolinkurl{james@intfxdx.com}}}
}

\maketitle{}
\vspace{-.5cm}
\begin{abstract}
This paper is centered on the random graph generated by a Doeblin-type coupling
of discrete time processes on a countable state space whereby when two paths meet, they merge.
This random graph is studied through a novel subgraph,
called a bridge graph, generated by paths started in
a fixed state at any time.
The bridge graph is made into a unimodular network by marking it and selecting a root
in a specified fashion.
The unimodularity of this network is leveraged to discern global properties of the larger Doeblin graph.
Bi-recurrence, i.e., recurrence both forwards and backwards in time,
is introduced and shown to be a key property in uniquely distinguishing 
paths in the Doeblin graph, and also a decisive property for Markov chains
indexed by $\Z$.
Properties related to simulating the bridge graph are also studied.
\end{abstract}

\vspace{.2cm}

\vspace{.2cm}
\vspace{-.3cm}
\tableofcontents{}

\section{Introduction}
The first incarnation of the Propp and Wilson~\cite{propp1996exact} coupling from the past (CFTP) algorithm
was designed to build a perfect sample from the stationary distribution $\pi$
of an irreducible, aperiodic, and positive recurrent Markov chain on a finite state space $S$.
It uses a Doeblin-type coupling of a family of copies of the Markov chain started
in all possible states at all possible times, whereby when two chains meet, they merge.
This coupling is represented with a random directed graph on $\Z \times S$ depicting the trajectories
of these Markov chains. Below, this random graph will be referred to as the \defn{Doeblin graph} of the chain.

Prior to this research, the study of this random graph has been
mostly a by-product of research on perfect simulation. In 1992--1993,
Borovkov and Foss~\cite{borovkov1992stochastically, borovkov1994two} laid out the framework of
stochastically recursive sequences (SRS), of which Markov chains are a special case,
and they proved the main results on the existence of a stationary
version of an SRS to which non-stationary versions converge in a certain sense.
The CFTP algorithm itself was introduced by Propp and Wilson in 1996 in~\cite{propp1996exact}
for obtaining samples from the stationary distribution of a Markov chain. 
The CFTP algorithm can be seen as a specialization of the general ideas of~\cite{borovkov1992stochastically} for SRS 
to the Markov case aiming at perfect simulation.
Foss and Tweedie~\cite{foss1998perfect} then gave a necessary and sufficient condition for the CFTP algorithm to converge a.s.
From 1996 to 2000, many 
papers~\cite{fill1997interruptible, propp1998coupling, murdoch1998exact, propp1998get, kendall1998perfect,
haggstrom1999exact, haggstrom1999characterization, moller1999perfect, kendall2000perfect, wilson2000couple, meng2000towards} 
investigated how to improve CFTP implementations or how to apply
CFTP or a CFTP-inspired algorithm to obtain a perfect sample from a particular
Markov chain's stationary distribution.
Of particular importance is Wilson's read-once CFTP algorithm~\cite{wilson2000couple}, which
allows CFTP to be done by only simulating forwards in time.
A review of perfect simulation in stochastic geometry up to that point is provided
in~\cite{moller2001review}.
Since then,~\cite{kendall2004geometric,connor2007perfect} showed that (possibly
impractical) generalizations of the CFTP 
algorithm can be applied under weaker conditions,
and~\cite{foss2003extended} gives a CFTP-like algorithm that applies even in the non-Markovian setting.

In this paper, focus is shifted away from finding an individual sample from the stationary distribution of a Markov chain,
and instead properties of the Doeblin graph as a whole are studied.
The SRS framework will be used, but, because the Markov case is a fundamental special case,
most sections will spell out what can be said in the Markov case.
The main tool of study is the theory of unimodular random (rooted) networks in the sense of Aldous and Lyons~\cite{aldous2007processes}.
Unimodular networks are rooted networks where, heuristically, the root is picked uniformly at random.
In order to generalize this concept for infinite networks, instead of picking
the root uniformly at random, the network is required 
to satisfy a mass transport principle.  The primary new object of study is the subgraph of bridges
between a fixed recurrent state, which is referred to as the \defn{bridge graph}
and is roughly inspired by the population process in~\cite{baccelli2018renewal}.
The subgraph is defined by looking at processes started at any time from this fixed state.
General setup and definitions of the Doeblin graph and the bridge graph are given in \Cref{sec:doeblin-graphs}.
Random networks and how to view subgraphs like the bridge graph as random networks are handled in \Cref{sec:randomnetworks}.
The main theorem is then proved in \Cref{sec:main-thm}.

\Cref{sec:unimodularizability-of-bridge} proves the main theorem, identifying the unimodular structure in the bridge graph.
\Cref{sec:i-f-component-properties} studies properties of the bridge graph that are inherited due to
its I/F component structure as a unimodular network.
Here I/F refers to the class of a component in the sense of the foil classification theorem in
unimodular networks in~\cite{baccellieternal}, which is reviewed in \Cref{sec:randomnetworks}.
The most interesting case is when $S$ is infinite and the Doeblin graph is connected. 
In this case (see \Cref{cor:ergodic-meft-unique-bi-recurrent}),
although there may be infinitely many bi-infinite paths in the Doeblin graph, there exists a unique
\defn{bi-recurrent} path, a bi-infinite path that visits every state infinitely often in the past, as well as in the future.
This unique path also has the property that the states in $S$ that the path traverses 
form a stationary version of the original Markov chain (or SRS), and hence
give samples from its stationary distribution.
Indeed, the original CFTP algorithm ultimately computes the time zero point on the bi-recurrent path.
By embedding Markov chains inside Doeblin graphs, 
bi-recurrence is also shown to be a decisive property for Markov chains indexed by $\Z$.
\Cref{thm:stationary-iff-birecurrent} shows that if a Markov chain $\Fam{X_t}_{t \in \Z}$ 
has an irreducible, aperiodic, and positive recurrent transition matrix, then $\Fam{X_t}_{t \in \Z}$
is stationary if and only if it is bi-recurrent for any (and hence every) state.
The I/F structure of a component leads to further useful qualitative properties discussed in \Cref{sec:other-i-f-properties}.
In reversed time, the bridge tree can be seen as a multi-type branching-like process where the types are
the elements of $S$, and for which there is at most one child of each type per generation.
The nodes in this branching process are either mortal (i.e., with finitely many descendants) or immortal (resp.\ infinitely
many). The mortal descendants of the nodes on the bi-infinite path form a stationary sequence of finite trees.
Mean values in these trees are linked to coupling times by mass-transport relations.
Finally, \Cref{sec:simulation-applications} gives results that are relevant to simulating the bridge graph, 
such as approximating the bridge graph by finite networks, and viewing the process of
vertical slices of the bridge graph as a Markov chain in its own right.
The final section gives several bibliographical comments, which make connections of the present research to other works. 
\section{The Doeblin Graph}\label{sec:doeblin-graphs}
\subsection{Definition}\label{sec:doeblin-graph-denf}
In this section, the Doeblin graph is constructed.
Fix a probability space $(\Omega, \F, \P)$, a countable state space $S$,
and a complete separable metric space $\Xi$ for the remainder of the document.
The first ingredient needed is a \defn{pathwise transition generator}, a function $h_{\mathrm{gen}}: S\times \Xi \to S$
that will be used for determining transitions between states of $S$.
Such an $\hgen$, combined with a \defn{driving sequence} $\Fam{\xi_t}_{t \in \N}$,
is used to give a pathwise representation of a stochastic process $\Fam{X_t}_{t \in \N}$ satisfying
\begin{equation}\label{eqn:srs}
    X_{t+1} := \hgen(X_t, \xi_t),\qquad t \geq 0.
\end{equation}
\Cref{eqn:srs} is the defining property of a stochastically recursive sequence (SRS)
in the sense of Borovkov and Foss~\cite{borovkov1992stochastically}.
If the driving sequence is taken to be i.i.d.\ and independent of $X_0$,
then $\Fam{X_t}_{t \in \N}$ is
a (discrete time) Markov chain with transition matrix $P = \Fam{p_{x,y}}_{x,y \in S}$ determined by
$p_{x,y} := \P(\hgen(x,\xi_0) = y)$ for each $x,y \in S$.
It is a classical result that, when $\Xi := [0,1]$,
all possible transition matrices $P$ can be achieved by choosing $\hgen$ and the distribution of $\xi_0$ accordingly
(c.f.\ Chapter 17 in~\cite{borovkov2013stochastic}).
Many processes in this paper will be indexed by $\Z$ or an interval of $\Z$ instead of just $\N$.
The pathwise transition generator $\hgen$ and a stationary and ergodic \emph{bi-infinite}
driving sequence
$\xi:= \Fam{\xi_{t}}_{t\in \Z}$, are fixed for the remainder of the document.
The notation for the transition matrix $P = \Fam{p_{x,y}}_{x,y \in S}$ is also
fixed for the remainder of the document, even when $\xi$ is not assumed to be i.i.d.

The space $\Z \times S$ should be thought of as time and space coordinates, with
$(t,x) \in \Z \times S$ being in state $x$ at time $t$.
The vertices and edges of a graph $\Gr$ will be written $V(\Gr)$ and $E(\Gr)$, and if
$V(\Gr) \subseteq \Z \times S$,
the vertices of $\Gr$
sitting at a particular time $t$ or in a particular state $x$ will be denoted, respectively, as
\begin{align}
    V_t(\Gr)  := \set{(s,y) \in V(\Gr) : s=t},\quad
    V^x(\Gr)  := \set{(s,y) \in V(\Gr) : y=x}.
\end{align}
Note that $V_t(\Gr)$ and $V^x(\Gr)$ are subsets $\Z \times S$, i.e.\ 
their elements have both a time component
and a space component.
If instead just states (elements of $S$) or just times (elements of $\Z$) are desired,
then the following are used instead
\begin{align}\label{defn:subscriptt}
    \Gr_t := \set{x \in S:  (t,x) \in V_t(\Gr)},\quad
    \Gr^x := \set{t \in \Z: (t,x) \in V^x(\Gr)}.
\end{align}

Then the \defn{Doeblin graph} $\G= \G(\hgen, \xi)$ is constructed as follows.
It has vertices $V(\G) := \Z \times S$.
The edges of $\G$ are determined by the 
\defn{follow map} $\fnext: V(\G) \to V(\G)$, which is a random map
giving directions of where each vertex
should move to in the next time step.
It is defined by
\begin{equation}
    \fnext(t,x) := (t+1,\hgen(x,\xi_{t})),\qquad (t,x) \in \Z \times S.
\end{equation}
That is, let the edges of $\G$ be drawn from each $(t,x) \in \Z \times S$ to $\fnext(t,x)$.
By saying a function $f:A \to B$ is a random map, it is meant that $f:A\times \Omega \to B$ is measurable and the second argument will be omitted.
Iterates of $\fnext$ are denoted by $\fnext^n$ for $n \geq 0$.
Thinking of each vertex in $\G$ as an individual, one may also interpret the follow
map as mapping each vertex to its parent vertex.

\subsection{Modeling}\label{sec:doeblin-graph-modeling}

When dealing with paths in $\G$, it will often be convenient to ignore the time coordinate
and focus only on the space coordinate.
If $\Fam{X_t}_{t \in I}$ is a stochastic process defined on $\Omega$ which takes values in $S$
and is such that $\Fam{t,X_t}_{t \in I}$ is a.s.\ a path in $\G$ over some fixed time interval $I\subseteq \Z$, then
$\Fam{X_t}_{t \in I}$ is called the \defn{state path (in $\G$)} corresponding to the path $\Fam{t,X_t}_{t \in I}$.
That is, there are two ways of looking at every route through $\G$:
as a path $\Fam{t,X_t}_{t \in I} \subseteq V(\G)$, or as a state path $\Fam{X_t}_{t \in I} \subseteq S$.

\begin{lemma}\label{lem:embed-srs-into-doeblin-graph}
    Let $I\neq\emptyset$ be an interval in $\Z$.
    Suppose that $\Fam{X_t}_{t \in I}$ is a stochastic process taking values in $S$ 
    a.s.\ satisfying the recurrence relation
    $X_{t+1} = \hgen(X_t, \xi_t)$ for each $\inf I \leq t < \sup I$, where $\hgen$ and $\Fam{\xi_t}_{t \in I}$ are the same as are used to define $\G$.
    Then $\Fam{X_t}_{t \in I}$ is a state path in $\G$.
\end{lemma}

\begin{proof}
    One must check that ${(t, X_t)}_{t \in I}$ is a.s.\ a path in $\G$.
    Fix $t \in I$.
    Since $V(\G) = \Z \times S$, $(t, X_t)$ is certainly a vertex of $\G$.
    If $t+1 \in I$ as well, one must check the edge $e$ from $(t, X_t)$ to $(t+1, X_{t+1})$ is a.s.\ an edge in $\G$.
    The edges of $\G$ are defined to be from each $(t,x) \in \Z \times S$ to $(t+1, \hgen(x,\xi_t))$,
    so the relation $X_{t+1} = \hgen(X_t, \xi_t)$ holding a.s.\ implies the edge $e$ is a.s.\ an edge of $\G$.
\end{proof}

In particular, \Cref{lem:embed-srs-into-doeblin-graph} says that any SRS whose driving sequence
is defined for all times in $\Z$ can be seen as living inside a Doeblin graph,
namely the one generated by its driving sequence and choosing $\hgen$ to be the same as in the definition of the SRS\@.

State paths started at a deterministic vertex will also be used heavily.
For the remainder of the document, 
let $F^{(t,x)}:= \Fam{F^{(t,x)}_s}_{s \geq t}$ be the \defn{state path in $\G$ started at time $t$ in state $x$},
i.e., $F^{(t,x)}$ is a re-indexing of the states traversed by $f_+$ defined by
\begin{align}
    (s, F^{(t,x)}_s) = \fnext^{s-t}(t,x),\qquad (t,x) \in \Z\times S,\quad s \geq t.
\end{align}
One has that $F^{(t,x)}$ is a version of the SRS or Markov chain started
in state $x$ with initial condition given at time $t$.
Generally speaking, throughout the paper, a parenthesized superscript, as in $F^{(t,x)}$, refers to a starting location.
For every $x \in S$, the distribution of $\Fam{F^{(t,x)}_{s+t}}_{s \geq 0}$ does not depend on $t$ because $\xi$ is stationary.
An example of a Doeblin graph and path of $F^{(t,x)}$ are drawn in~\Cref{fig:meft-example}.

\begin{figure}[t!]
    \centering
    \includegraphics[width=\linewidth]{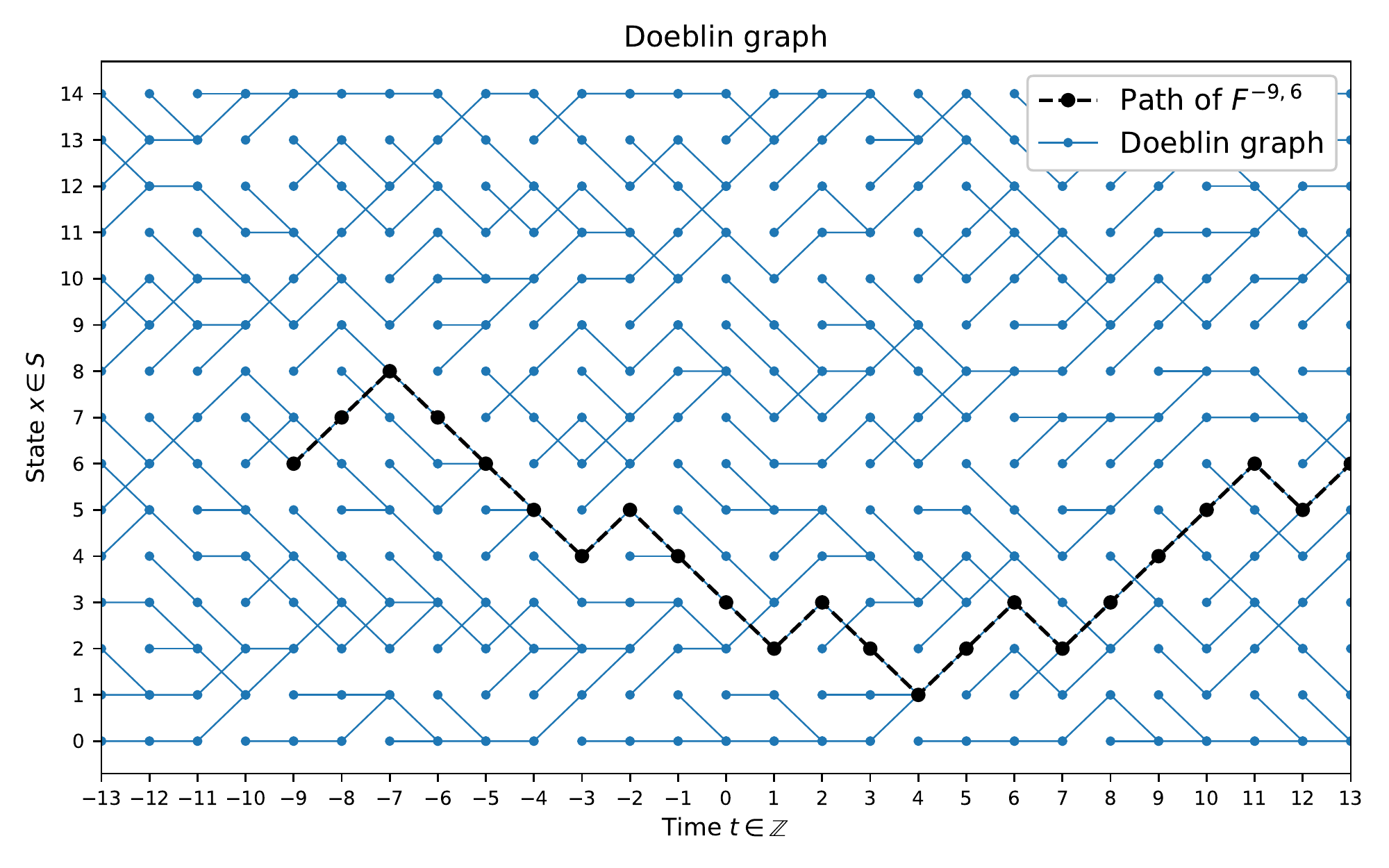}
    \caption{An example of a Doeblin graph with the path corresponding to the state path $F^{(t,x)}$ distinguished.
        All edges are directed from left to right.}%
    \label{fig:meft-example}
\end{figure}

It has already been noted (see~\cite{borovkov2013stochastic}) that a Markov chain $\Fam{X_t}_{t \in \N}$
with any given desired transition matrix can be constructed as an SRS with i.i.d.\ driving sequence.
The following is an analogous result saying that any Markov chain $\Fam{X_t}_{t \in \Z}$ may be realized as
a state path in a Doeblin graph with i.i.d.\ driving sequence.
Note here that the time index set is all of $\Z$, not just $\N$.

\begin{theorem}\label{thm:embed-into-general-meff}
    Suppose that $\Fam{X_t}_{t \in \Z}$ 
    is a Markov chain
    with transition matrix $P$
    on some probability space,
    where $P$ is the same as was defined for the Doeblin graph $\G$.
    Also suppose the driving sequence $\xi$ is i.i.d.
    Then there is a probability space $(\Omega',\F',\P')$ and $\Fam{X_t'}_{t \in \Z}\sim \Fam{X_t}_{t \in \Z}$ on $\Omega'$ such that
    $\Fam{X_t'}_{t \in \Z}$ is state path in $\G'$, where $\G'$ is the Doeblin graph generated
    by some i.i.d.\ driving sequence $\xi'=\Fam{\xi_{t}'}_{t\in \Z} \sim \xi$
    in $\Omega'$ with pathwise transition generator $\hgen$.
    Moreover, for each $t \in \Z$, $X'_t$ is independent of $\Fam{\xi_s'}_{s \geq t}$.
\end{theorem}

\begin{postponedproof}
    Consider a probability space housing independent copies of
    $\Fam{X_t}_{t \in \Z}$ and $\G$. Then consider
    for each $t \in \Z$ the state path in $\G$ started at $X_t$.
    The distributions of these state paths
    determine a consistent set of finite dimensional distributions for
    the desired pair of processes
    $(\Fam{X'_t}_{t \in \Z}, \Fam{\xi'_t}_{t \in \Z})$.
    By the Kolmogorov extension theorem, the result follows.
    The full proof of~\Cref{thm:embed-into-general-meff} is given in the
	appendix.
\end{postponedproof}

\subsection{Basic Properties}

Plainly, $\G$ is acyclic as an undirected graph because all outgoing edges point forward one unit in time and each vertex has only one outgoing edge.
When $\G$ is a.s.\ connected, it is called a \defn{Doeblin Eternal Family Tree} or a \defn{\deft{}} for short.
More generally, $\G$ may have up to countably many components and is referred to as a \defn{Doeblin Eternal Family Forest} or \defn{\deff{}}.
The \eft{} and \eff{} terminology is inspired by~\cite{baccellieternal} and the word eternal refers to the fact that every vertex of $\G$ has a unique outgoing edge.
That is, there is no individual that is an ancestor of all other individuals.
An \eff{} is a more general object than an \eft{}, i.e.\ an \eff{} may also be an \eft{}.

If the driving sequence $\xi$ is i.i.d., so that the state paths $F^{(t,x)}$ for each $(t,x) \in \Z \times S$ are Markov chains,
then say that $\G$ is \defn{Markovian}.
If $\xi$ is such that for each $t \in \Z$, $\Fam{\fnext(t,x)}_{x \in S}$ is an independent family,
then $\G$ is said to have \defn{vertical independence}.
If $\G$ is Markovian and has vertical independence, then say that $\G$ has \defn{fully independent transitions}.

Some later results are only valid for \efts{}, so the following result
gives an easy case when $\G$ can be shown to be connected.

\begin{proposition}\label{prop:teaser-indep-transitions-components}
    Suppose $\G$ has fully independent transitions, and $P$ is irreducible and positive recurrent with period $d$.
    Then a.s.\ $\G$ has $d$ components. 
    In particular, if $P$ is irreducible, aperiodic, and positive recurrent, then $\G$ is an \eft{}.
\end{proposition}

\begin{postponedproof}
    The case of a general $d$ is reduced to $d=1$ by
    viewing the chain only every $d$ steps and with state space restricted to 
    one of the $d$ classes appearing in a cyclic decomposition of the state space.
    Consider the state paths in $\G$ started at $(0,x)$ and $(0,y)$ for any two $x,y$.
    Strictly before hitting the diagonal, 
    the pair of state paths has the same distribution
    as a product chain, i.e.\ two independent copies of the chain with one started at $x$ and the other at $y$.
    The product chain is irreducible, aperiodic, and positive recurrent, and
    therefore a.s.\ hits the diagonal, showing the state paths
    started at $(0,x)$ and $(0,y)$ eventually merge.
    The full proof of~\Cref{prop:teaser-indep-transitions-components} is given in the
	appendix.
\end{postponedproof}

A $\xi$-measurable subgraph $\Gr = \Gr(\Fam{\xi_{t}}_{t \in \Z})$ of $\G$ is called \defn{shift-covariant} if, for all $s \in \Z$,
$\Gr(\Fam{\xi_{t+s}}_{t\in \Z})$ is a.s.\ the time-translation of $\Gr$ by $-s$.
Say a state path $\Fam{X_t}_{t \in \Z}$ is \defn{shift-covariant} if the corresponding path in $\G$ is shift-covariant.
In other words, if the driving sequence $\xi$ is translated by some amount $s$ in time, then shift-covariant
objects are also translated in time by the same amount.
Let $E \in \F$ be $\xi$-measurable, 
say $1_E = g(\Fam{\xi_{t}}_{t\in \Z})$.
Say that $E$ is \defn{shift-invariant} if $g(\Fam{\xi_{t}}_{t\in \Z})= g(\Fam{\xi_{t+1}}_{t\in \Z})$ a.s.
That is, shift-invariant events are those events whose occurence is unaffected by time translations of the driving sequence $\xi$.
One has that $\P(E) \in \set{0,1}$ for all shift-invariant events $E$ due to the ergodicity of $\xi$.
All of the following are shift-invariant and hence happen with probability zero or one:
$\G$ is locally finite,
$\G$ contains no cycles,
$\G$ is connected,
$\G$ has exactly $n \in \N \cup\set{\infty}$ components,
$\G$ contains exactly $n \in \N \cup\set{\infty}$ bi-infinite paths.
Generally it will be obvious whether an event is shift-invariant.

When $\G$ is a Markovian, one needs to be cautious that not all state paths in $\G$ are Markov chains with transition matrix $P$.

\begin{example}\label{ex:nonproperstatepath}
    Let $S:=\Z$ and suppose $\G$ has fully independent transitions with $p_{x,x-1} = p_{x,x} = p_{x,x+1} = \frac{1}{3}$ for all $x \in S$.
    Choose $X_0$ to be the smallest element of $\Z$ (in some well-ordering of $\Z$) 
    such that
    $F^{(0,X_0)}_1 = F^{(0,X_0)}_2$.
    In this case, a.s.\ $X_1 = X_2$, so $\Fam{X_t}_{t \in \N}$ is not even Markovian.
\end{example}

The problem with the path in the previous example is that it looks into the future.
Namely, the value of $X_0$ depends on information at time $1$ and time $2$.
To exclude state paths like those in \Cref{ex:nonproperstatepath},
the notion of properness is introduced.
For a nonempty interval $I$ of $\Z$, if for each $t \in I$, $X_t$ is independent of $\Fam{\xi_{s}}_{s \geq t}$,
then $\Fam{X_t}_{t \in I}$ is called a \defn{proper} state path.
In the Markovian case, if $I$ has a minimum element $t_0$, then to show that a
state path $\Fam{X_t}_{t \in I}$ is proper it is sufficient that $X_{t_0}$ is independent of 
$\Fam{\xi_{s}}_{s \geq t_0}$ because for any $s \in \N$, $X_{t_0+s}$ is measurable with respect to the $\sigma$-algebra
generated by $X_{t_0}$ and $\xi_{t_0},\ldots,\xi_{t_0 + s-1}$.
Unlike general state paths in $\G$, proper state paths inherit a Markov transition structure.

\begin{lemma}\label{lem:proper-state-path-is-markov-chain}
    Suppose $\G$ is Markovian.
    If $\Fam{X_t}_{t \in I}$ is a proper state path in $\G$ over a nonempty interval $I\subseteq \Z$,
    then $\Fam{X_t}_{t \in I}$ is a Markov chain with transition matrix $P$.
\end{lemma}

\begin{proof}
    Fix $t < \sup I$.
    Let  $E:=\set{X_{t}=x_t,\ldots, X_{t-k}=x_{t-k}}$ be given with $k \in \N$ such that $t-k \geq \inf I$, and 
    $x_t,\ldots,x_{t-k} \in S$.
    Note that whether $E$ occurs is a function of $X_{t-k}$ and $\xi_{t-k},\ldots, \xi_{t-1}$, 
    so the fact that $X_{t-k}$ is independent of $\Fam{\xi_s}_{s \geq t-k}$ and the fact that $\xi$ is 
    i.i.d.\ imply that $E$ is independent of $\Fam{\xi_{s}}_{s \geq t}$. 
    Then for any $x \in S$,
    \begin{align*}
        \E[1_{\set{X_{t+1}=x}} 1_{E}]
        &=\E[1_{\set{h(x_t, \xi_{t})=x}} 1_{E}]\\
        &=\P(h(x_t, \xi_{t})=x) \P(E)\\
        &= p_{x_t,x} \P(E)\\
        &= \E[p_{X_t,x} 1_E].
    \end{align*}
    Since $E$ was an arbitrary cylinder set, it follows that for all $x\in S$,
    \[ 
        \P(X_{t+1} = x \mid \Fam{X_s}_{s \in I, s\leq t}) = p_{X_{t},x}.
    \]
    Thus $\Fam{X_s}_{s \in I}$ is a Markov chain with transition matrix $P$.
\end{proof}

\subsection{Connections with CFTP}

Consider the following structural result that will be expanded upon
in \Cref{sec:bi-recurrence}.
It is a special case of \Cref{prop:existence-of-bi-rec-paths} and \Cref{cor:ergodic-meft-unique-bi-recurrent}, which
will be proved later.

\begin{proposition}\label{prop:mainthm-teaser}
    Suppose $\G$ is Markovian, and 
    that $P$ is irreducible, aperiodic, and positive recurrent.
    Then a.s.\ in every component of $\G$ there exists a unique bi-infinite path that visits every state in $S$ infinitely often
    in the past.
    All other bi-infinite paths in $\G$ do not visit any state infinitely often in the past.
    If $\G$ is an \eft{}, then with $\beta_t$ denoting the state at time $t$ of the unique bi-infinite path visiting every state infinitely often in the past,
    one has that $\Fam{\beta_t}_{t \in \Z}$ is a stationary Markov chain with transition matrix $P$, so that $\beta_t \sim \pi$ for all $t \in \Z$, where $\pi$
    is the invariant distribution for $P$.
\end{proposition}

The main result of the original Propp and Wilson paper can be translated into
the language of \deffs{} and summarized as follows.
The reader is encouraged to ponder what it says about the structure of $\G$, and in doing so
one sees that is has much the same spirit as~\Cref{prop:mainthm-teaser}.

\begin{proposition}[Perfect Sampling~\cite{propp1996exact}]\label{prop:PW-CFTP}
    If $S$ is finite and $\G$ is Markovian and an \eft{} (which, since $S$ is
    finite, necessitates that $P$ is irreducible and aperiodic),
    then there is an a.s.\ finite time $\tau$ such that all paths in $\G$ started at any time $t \leq -\tau$
    have merged by time $0$, all reaching a common vertex $(0, \beta_0)$.
    Moreover, $\beta_0 \sim \pi$, where $\pi$ is the stationary distribution of $P$, and there is an algorithm $A$
    that a.s.\ terminates in finite time returning $\beta_0$.
\end{proposition}

\begin{remark}
    In fact, the $\beta_0$ appearing in \Cref{prop:PW-CFTP}
    and the $\beta_0$ appearing in \Cref{prop:mainthm-teaser} are the same.
    That is, the perfect sampling algorithm $A$ is ultimately computing the
    point in $\G$ on the unique bi-infinite path and returning its state.
    This can be seen by the fact that, since all paths started at time $-\tau$
    reach the common vertex $(0,\beta_0)$, any bi-infinite path in $\G$
    must also pass through $(0,\beta_0)$.
    However, what is notably absent in \Cref{prop:mainthm-teaser} is any mention of an algorithm to compute $\beta_0$.
    Whether such an algorithm exists in general is not studied in the present research.
\end{remark}

\subsection{Bridge Graphs}\label{sec:MBEFFs}
The primary tool used in this document will be the theory of unimodular networks in the sense of~\cite{aldous2007processes}.
Local finiteness is essential in the theory of unimodular networks, 
but the Doeblin graph $\G$ may not be locally finite,
as the following result shows.

\begin{proposition}
    If $\sum_{x\in S} p_{x,y} < \infty$ for all $y \in S$, then $\G$ is a.s.\ locally finite.
    If $\G$ has fully independent transitions and for some $y\in S$, $\sum_{x \in S} p_{x,y} = \infty$,
    then $\G$ is a.s.\ not locally finite.
\end{proposition}

\begin{proof}
    Both statements follow from the Borel-Cantelli lemmas.
    That is, for any fixed $(t,y) \in \Z\times S$,
    if $\sum_{x \in S}p_{x,y} < \infty$, then a.s.\ one has that only finitely many
    of the events $\set{\fnext(t-1,x) = (t,y)}_{x \in S}$ occur, showing $(t,y)$
    has finite in-degree, and hence finite degree, in $\G$.
    On the other hand, if $\G$ has fully independent transitions and
    for some fixed $(t,y) \in \Z \times S$ one has $\sum_{x \in S}p_{x,y} = \infty$,
    then a.s.\ infinitely many of the events $\set{\fnext(t-1,x) = (t,y)}_{x \in S}$ occur,
    so that $(t,y)$ has infinite degree.
\end{proof}

The remedy taken here is to instead concentrate on particular subgraphs of $\G$.
In this section, subgraphs are introduced that are locally finite under a positive recurrence assumption 
and turn out to have nice properties
when considered as random networks.

For each $(t,x) \in \Z\times S$, and each $y \in S$, let
\begin{align}
    \tau^{(t,x)}(y) :=\inf\set{s > t: F^{(t,x)}_s = y}, \qquad \sigma^{(t,x)}(y) := \tau^{(t,x)}(y) - t
\end{align}
be, respectively, the \defn{return time} and \defn{time until return} of
$F^{(t,x)}$ to $y$.
The word return is used even when $y \neq x$, in which case it may be that $F^{(t,x)}$
is not part of a state path that has visited $y$ before time $t$.
Note that the distribution of $\sigma^{(t,x)}(y)$ does not depend on $t$ because $\xi$
is stationary.
Call a state $x \in S$ \defn{positive recurrent} if $\E[\sigma^{(0,x)}(x)] < \infty$ or \defn{recurrent} if $\sigma^{(0,x)}(x) < \infty$ a.s.
In the Markovian case these are the usual definitions.
If a state $x \in S$ is recurrent, then indeed for every $t \in \Z$, $F^{(t,x)}$ visits
$x$ infinitely often. 

For each fixed $x \in S$, consider the subgraph $\B(x)$ of $\G$ of all paths 
starting from state $x$ at any time.
That is, $\B(x)$ is the subgraph of $\G$ with
\begin{align}\label{eq:Bxvertices}
    V(\B(x)) 
    := \bigcup_{t \in \Z}\set{(s,F^{(t,x)}_s) : s \geq t}
    = \bigcup_{t \in \Z}\set{\fnext^n(t,x) : n \geq 0}.
\end{align}
Call $\B(x)$ the \defn{bridge graph for state $x$}
and refer to it as either a \defn{\beff{}} or \defn{\beft{}} depending on whether it is a forest or a tree.
Note that one of these possibilities happens with probability $1$ because the number of components in $\B(x)$
is shift-invariant.

\begin{assumption}
    For the remainder of the document, 
    assume there exists a positive recurrent state $x^* \in S$, which is fixed,
    and the notation $\B := \B(x^*)$ refers to the bridge graph for state $x^*$.
\end{assumption}

An example bridge graph appears in \Cref{fig:meft-and-bridge-example}.
\begin{figure}[t!]
    \centering
    \includegraphics[width=\linewidth]{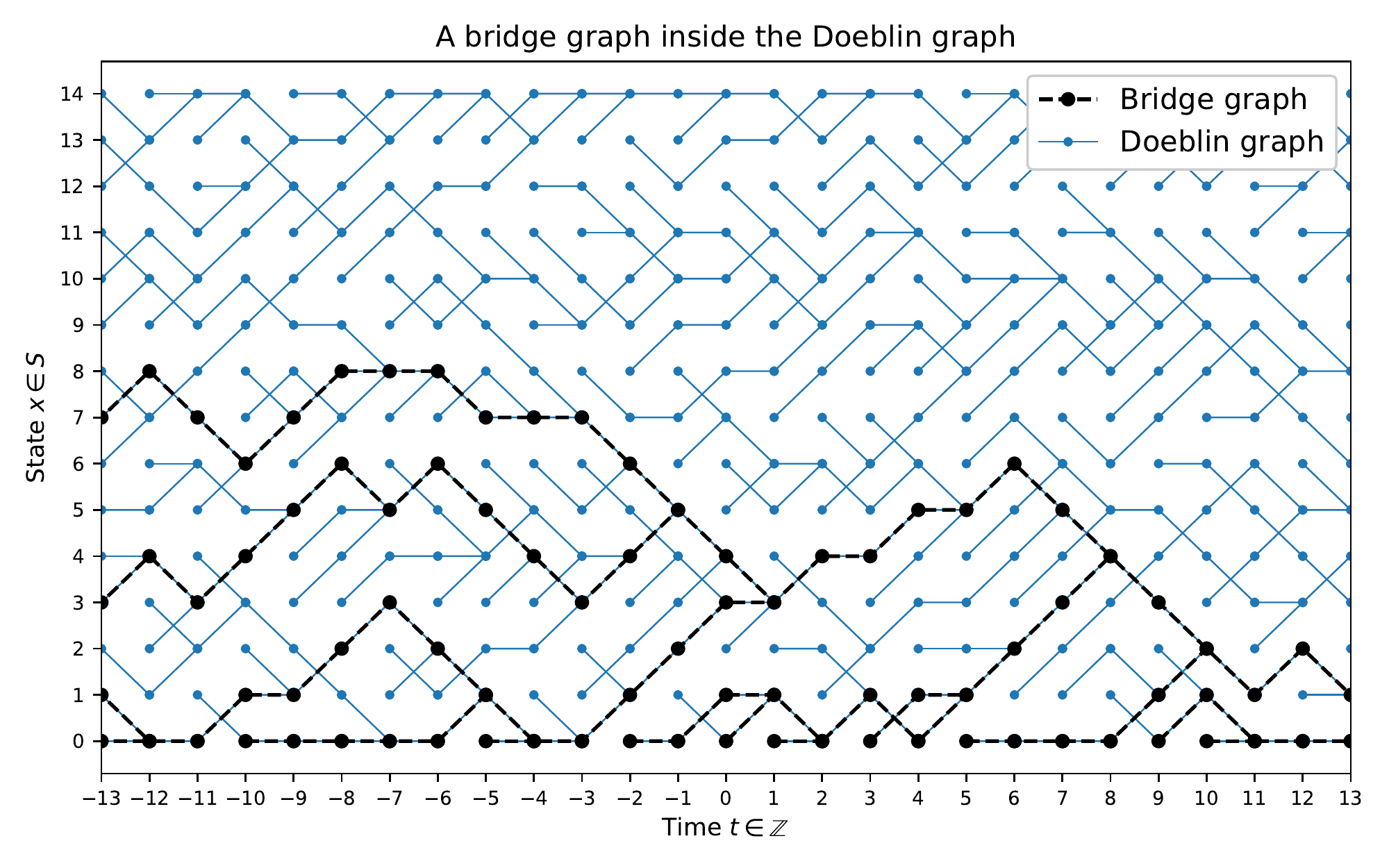}
    \caption{An example bridge graph, in this case for state $x^* = 0$, sitting inside the Doeblin graph.}%
    \label{fig:meft-and-bridge-example}
\end{figure}
Equivalently, $\B$ can be described
in terms of descendants of vertices, viewing directed edges in $\G$ as pointing from
a vertex to its parent.
For each $(t,y) \in \Z \times S$, define the \defn{descendants} of $(t,y)$ 
in $\G$ to be
\begin{align}
    D^{(t,y)} &:= \set{(s,x) \in \Z \times S : F^{(s,x)}_t = y}.
\end{align}
Then $\B$ is also the subgraph of $\G$ with
\begin{align*}
    V(\B) = \set{(t,y) \in \Z \times S : \exists s,\, (s,x^*) \in D^{(t,y)}}.
\end{align*}
That is, $\B$ is the subgraph of $\G$ generated by vertices that have some descendant in state $x^*$.
In particular, recalling~\eqref{defn:subscriptt},

\begin{align}
    y \in \B_t \iff \exists s,\, x^* \in D^{(t,y)}_s,\qquad (t,y) \in \Z \times S.
\end{align}
\Cref{lem:G-eft-implies-B-eft} shows that if $\G$ is a.s.\ connected, then $\B$ is too.

\begin{lemma}\label{lem:G-eft-implies-B-eft}
    If $u,v \in V(\B)$ are in the same component of $\G$, then they are in the same component of $\B$.
    In particular, if $\G$ is an \eft{}, then $\B$ is an \eft{}.
\end{lemma}

\begin{proof}
    Consider times $s,t \in \Z$.
    Suppose $(s,x^*)$ and $(t,x^*)$ are in the same component of $\G$.
    Then $F^{(s,x^*)}$ and $F^{(t,x^*)}$ meet at some point.
    But, by definition, the paths of $F^{(s,x^*)}$ and $F^{(t,x^*)}$ are included in $\B$.
    Hence $(s,x^*)$ and $(t,x^*)$ are in the same component of $\B$.
    Now if $u,v \in V(\B)$ are in the same component of $\G$,
    $u$ is in the same component in $\G$ as some $(s,x^*)$ and $v$ is in the same component of $\G$ as some $(t, x^*)$,
    and $(s,x^*)$ and $(t, x^*)$ are in the same component of $\B$ by the previous part.
    Hence $u,v$ are in the same component of $\B$.
\end{proof}

The condition that $\B$ is an \eft{} is equivalent to 
strong coupling convergence (defined and studied in~\cite{borovkov1992stochastically, borovkov1994two, foss2003extended})
of $F^{(0,x^*)}$ to a stationary version of the SRS\@.
However, simple conditions for $\B$ to be an \eft{} are not known
outside of the Markovian case, where \Cref{prop:teaser-indep-transitions-components}
showed that if $P$ is irreducible, aperiodic, and positive recurrent,
then $\G$ is an \eft{}.
Another (not neccessarily easy to check) condition for $\B$ to be an \eft{}
will be given in \Cref{cor:B-has-finitely-many-components}.

The main tool used in this paper is unimodularity of random networks.
The first form of unimodularity used is stationarity, i.e., the
unimodularity of the deterministic network 
$\Z$ rooted at $0$ and with neighboring integers connected.
Unimodularity of $\Z$ gives a helpful way to reorganize proofs based on stationarity in terms of transporting mass
between different times.
Recall that a (measurable) group action $\theta:\Z\times \Omega \to \Omega$ of $\Z$ on $\Omega$ is called $\P$-invariant
if $\P(\theta_t \in \cdot) = \P$ for all $t \in \Z$. 
The shift operator on $\Xi ^\Z$ is an example of such an action.

\begin{lemma}[Mass Transport Principle for $\Z$]\label{lem:mtpforZ}
    Suppose $w:\Z\times \Z \to \R_{\geq 0}$ is a random map.
    Also suppose $\theta:\Z \times \Omega \to \Omega$ is a $\P$-invariant $\Z$-action
    on $\Omega$, and that the two are compatible in the sense that $w(s,t)\circ \theta_r = w(s+r,t+r)$ almost surely for each $s,t,r \in \Z$.
    Then with $w^+:= \sum_{t \in \Z} w(0,t)$ and $w^- := \sum_{s \in \Z} w(s,0)$, one has
    \begin{align}
        \E[w^+] = \E[w^-].
    \end{align}
\end{lemma}

\begin{proof}
    One calculates
    \[
        \E[w^+] = \sum_{t \in \Z} \E[w(0,t)] 
        = \sum_{t \in \Z} \E[w(0,t) \circ \theta_{-t}] 
        = \sum_{t \in \Z} \E[w(-t,0)]
        = \E[w^-]
    \]
    as desired.
\end{proof}

The mass transport principle for $\Z$ immediately gives the following.

\begin{proposition}\label{prop:mbeff-is-locally-finite}
    For all $t \in \Z$, $\E[\#\B_t] \leq \E[\sigma^{(0,x^*)}(x^*)]$.
    In particular, $\B$ is a.s.\ locally finite, even if $\G$ itself is not.
\end{proposition}

\begin{proof}
    Without loss of generality, $\Omega$ is the canonical space $\Xi^{\Z}$,
    with the driving sequence $\Fam{\xi_{t}}_{t \in \Z}$ being coordinate maps.
    Then $\theta:\Z\times \Omega \to \Omega$ defined by 
    $\theta_s(\Fam{\xi_{t}}_{t\in \Z}) :=  \Fam{\xi_{s+t}}_{t\in \Z}$
    is a $\P$-invariant measurable $\Z$-action on $\Omega$.
    Choose the mass transport $w(s,t) := 1_{\set{\sigma^{(s,x^*)}(x^*) > t-s > 0}}$.
    The fact that one has $\sigma^{(s,x^*)}(x^*) \circ \theta_r = \sigma^{(s+r,x^*)}(x^*)$ for all $s,t,r \in \Z$
    implies $w$ is compatible with $\theta$.
    Then $w^+ = \sigma^{(0,x^*)}(x^*) -1$, and $w^- = \#\set{s<0: \sigma^{(s,x^*)}(x^*) > |s|} \geq \#\B_0-1$,
    where this inequality follows from the fact that for every $y \in \B_0 \setminus \set{x^*}$,
    there is $s<0$ such that $\sigma^{(s,x^*)}(x^*) > |s|$ and $F^{(s,x^*)}_0 = y$.
    Thus the mass transport principle for $\Z$ gives $\E[\sigma^{(0,x^*)}(x^*)-1] \geq \E[\#\B_0-1]$, from which the result follows.
\end{proof}

The proof style of \Cref{prop:mbeff-is-locally-finite} may be repeated in many different ways and the boilerplate
setup of the proof can be mostly omitted once one understands the flow of the proof.
The shortened version of the proof of \Cref{prop:mbeff-is-locally-finite} is given to exemplify how much can be omitted
without losing the main idea.

\begin{proof}[Proof (shortened)]
    Let the mass transport $w(s,t)$ send mass $1$ from $s$ to all times $t$ strictly after $s$
    and strictly before $F^{(s,x^*)}$ returns to $x^*$.
    Then $w^+ = \sigma^{(0,x^*)}(x^*)-1$ and $w^- = \#\set{s < 0: \sigma^{(s,x^*)}(x^*) > |s|} \geq \#\B_0 - 1$,
    where this inequality follows from the fact that for every $y \in \B_0 \setminus \set{x^*}$,
    there is $s<0$ such that $\sigma^{(s,x^*)}(x^*) > |s|$ and $F^{(s,x^*)}_0 = y$.
    The mass transport principle finishes the claim.
\end{proof}

One now sees the versatility of using even the simplest form of unimodularity.
A list of mass transports and the results they give, all by following the same proof style,
appears in \Cref{sec:list-of-mass-transports} in the appendix.
Some of the mass transports give new results, and others recover well-known results,
such as fact that $\pi(y)/\pi(x^*)$ is the expected number of visits of a Markov chain started at $x^*$
to $y$ before returning to $x^*$, and $1/\pi(x^*)$ is the expected return time of a Markov chain started at $x^*$
to return to $x^*$, where $\pi$ is the invariant distribution for the Markov chain.
The next section reviews the more general theory of random networks and unimodularity, then shows how to embed
subgraphs of $\G$ as random networks, so that eventually one may find a unimodular structure inside $\G$.
\section{Random Networks}\label{sec:randomnetworks}
\subsection{Definition and Basic Properties}
See~\cite{aldous2007processes,khezeli2017shift} for a more thorough review of random networks than what is provided here.
A \defn{network} is a graph
$\Gr=(V(\Gr),E(\Gr))$ equipped with a complete separable metric
space $(\Xi_\Gr, d_{\Xi_\Gr})$ called the \defn{mark space} and two maps from $V(\Gr)$ and $\set{(v,e): v \in V(\Gr), e \in E(\Gr), v \sim e}$
to $\Xi_\Gr$, where $\sim$ is used for adjacency of vertices or edges.
The image of $v$ (resp.\ $(v,e)$) in $\Xi_\Gr$ is called its \defn{mark},
which is extra information associated to the vertex (resp.\ edge).
The mark of $(v,e)$ may also be thought of as the mark of $e$ considering it to be a directed edge with initial vertex $v$.
The graph distance between $v$ and $w$
is denoted $d_\Gr(v,w)$.
Unless explicitly mentioned otherwise, networks are assumed to be nonempty, locally finite, and connected.

An \defn{isomorphism} between two networks with the same mark space is a graph isomorphism
that also preserves the marks.
A \defn{rooted network} is a pair $(\Gr,o)$ in which $\Gr$ is a network and $o$ is a 
distinguished vertex of $\Gr$ called the \defn{root}.
An \defn{isomorphism} of rooted networks is a network isomorphism that takes the root of one network to
the root of the other.
Similar definitions apply to doubly rooted networks $(\Gr,o,v)$.
For convenience, from now on consider only networks with mark space $(\Xiuniv, d_{\Xiuniv})$,
where $\Xiuniv$ is some fixed uncountable complete separable metric space, such as $\N^\N$ or the Hilbert cube,
since all possible mark spaces are homeomorphic to a subset of such a $\Xiuniv$.
Let $\mathcal{G}$ denote the set of isomorphism classes of nonempty, locally finite, connected networks, and 
let $\mathcal{G}_{*}$ (resp.\ $\mathcal{G}_{**}$) be the set of isomorphism classes of singly (resp.\ doubly) rooted
networks of the same kind.
The isomorphism class of a network $\Gr$ (resp.\ $(\Gr,o)$, or $(\Gr,o,v)$) is denoted by $[\Gr]$ (resp.\ $[\Gr,o]$ or $[\Gr,o,v]$).

The sets $\mathcal{G}_{*}$ and $\mathcal{G}_{**}$ are equipped with natural
metrics
making them complete separable metric spaces
(cf.~\cite{aldous2007processes}).
The distance $d_{\mathcal{G}_*}([\Gr_1, o_1], [\Gr_2, o_2])$ between the isomorphism classes of $(\Gr_1,o_1)$ and $(\Gr_2,o_2)$ is $1/(1+\alpha)$, where
$\alpha$ is the supremum of those $r>0$ such that there is a rooted isomorphism of the balls
of graph-distance $\lfloor r \rfloor$ around the roots of $\Gr_1, \Gr_2$ such that each
pair of corresponding marks has distance less than $1/r$.
The distance on $\mathcal{G}_{**}$ is defined 
similarly and the projections $[\Gr,o,v] \mapsto [\Gr,o]$ and $[\Gr,o,v] \mapsto [\Gr,v]$ are continuous.

A \defn{random (rooted) network} is a random element in $\mathcal{G}_{*}$ equipped with its Borel $\sigma$-algebra
$\mathcal{B}(\mathcal{G}_{*})$.
A random network $[\bs{\Gr},\bs{o}]$ is called \defn{unimodular} if for all measurable 
$g:\mathcal{G}_{**} \to \R_{\geq 0}$, the following \defn{mass transport principle} is satisfied:
\begin{align}\label{eq:mtp}
    \E \sum_{v \in V(\bs{\Gr})} g[\bs{\Gr},\bs{o},v] = \E \sum_{v \in V(\bs{\Gr})} g[\bs{\Gr},v,\bs{o}].
\end{align}
Heuristically, the root of a unimodular network is picked uniformly at random from its vertices.
However, since there is no uniform distribution on an infinite set of vertices,
the mass transport principle~\eqref{eq:mtp} is used in lieu of requiring the root to be picked uniformly at random.
One should take care to note that the sums in the previous equation depend
only on the isomorphism class $[\bs{\Gr},\bs{o}]$ and not which representative is used.

Next, the notions of covariant vertex-shifts, foils, connected components, and the cardinality classification of components
of a unimodular network are reviewed.
See~\cite{baccellieternal} for a reference on these concepts.
A \defn{(covariant) vertex-shift} is a map $\Phi$ which associates to each network $\Gr$
a function $\Phi_\Gr: V(\Gr) \to V(\Gr)$ such that $\Phi$ commutes with network isomorphisms
and the function $[\Gr,o,v] \to 1_{\set{\Phi_\Gr(o)=v}}$ is measurable on $\mathcal{G}_{**}$.
For a vertex-shift $\Phi$, define two equivalence relations on each network $\Gr$
by saying $u,v \in V(\Gr)$ are in the same \defn{$\Phi$-foil} if $\Phi_\Gr^n(u) = \Phi_\Gr^n(v)$ for some $n \in \N$,
or in the same \defn{$\Phi$-component} if $\Phi_\Gr^n(u) = \Phi_\Gr^m(v)$ for some $n,m \in \N$.
Two vertices are in the same $\Phi$-component if their forward orbits under $\Phi$ intersect,
whereas they are in the same $\Phi$-foil if, after some finite number of applications of $\Phi$,
the vertices meet.
The \defn{$\Phi$-graph of $\Gr$} is the graph drawn on $\Gr$ with vertices $V(\Gr)$ and edges from each $v \in V(\Gr)$ to $\Phi_\Gr(v)$.
The following is a special case of the classification theorem appearing in~\cite{baccellieternal}.

\begin{theorem}[Foil Classification in Unimodular Networks~\cite{baccellieternal}]\label{thm:foil-classification-theorem}
    Let $[\bs{\Gr},\bs{o}]$ be a unimodular network and $\Phi$ a vertex-shift.
    Almost surely, every vertex has finite degree in the $\Phi$-graph of $\bs{\Gr}$.
    In addition, each component $C$ of the $\Phi$-graph of $\bs{\Gr}$
    falls in one of the following three classes:
    \begin{enumerate}
        \item Class F/F\@: $C$ and all its foils are finite, and there is a unique cycle in $C$.
        \item Class I/F\@: $C$ is infinite but all its foils are finite, there are no cycles in $C$, and there is a unique bi-infinite path in $C$.
        \item Class I/I\@: $C$ is infinite and all its foils are infinite, and there are no cycles or bi-infinite paths in $C$.
    \end{enumerate}
\end{theorem}

The last tool needed from~\cite{baccellieternal} is the so-called no infinite/finite inclusion lemma, which is used heavily
in the proof of \Cref{thm:foil-classification-theorem}.
To state it, the following definitions are needed.
A \defn{covariant subset (of the set of vertices)} is a map $C$ which associates
to each network $\Gr$ a set $C_\Gr\subseteq V(\Gr)$ such that $C$ commutes with network isomorphisms, and such that $[\Gr,o]\mapsto 1_{\set{o \in C_\Gr}}$
is measurable.
A \defn{covariant (vertex) partition} is a map $\Pi$ which associates to all networks $\Gr$ a partition $\Pi_\Gr$
of $V(\Gr)$ such that $\Pi$ commutes with network isomorphisms, and such that
the (well-defined) subset $\set{[G,o,v]: v \in \Pi_G(o)} \subseteq \mathcal{G}_{**}$ is measurable,
where $\Pi_\Gr(o)$ denotes the partition element in $\Pi_\Gr$ containing $o$.
Then one has the following.

\begin{lemma}[No Infinite/Finite Inclusion~\cite{baccellieternal}]\label{lem:no-infinite-finite-inclusion}
    Let $[\bs{\Gr},\bs{o}]$ be a unimodular network, $\Pi$ a covariant partition, and $C$
    a covariant subset.
    Almost surely, there is no infinite element $E$ of $\Pi_{\bs{\Gr}}$ such that $E \cap C_{\bs{G}}$
    is finite and nonempty.
\end{lemma}

\subsection{Embedding Subgraphs of the Doeblin Graph as Random Networks}

In order to view a subgraph of $\G$ as a random network, one must ensure the subgraph 
is nonempty, locally finite, connected, and a root $\bs{o}$ has been suitably chosen.
Since the vertices of $\G$ come from the fixed countable space $\Z \times S$, the following setup will help to verify all the technicalities.

Let $V:= \Z \times S$.
Suppose that
\begin{equation}
    \bs{\Gr}:\Omega \to \set{0,1}^V \times \Xi^V \times \set{0,1}^{V\times V} \times \Xi^{V \times V} 
    =: (f_V, \xi_V, f_E, \xi_E)
\end{equation}
is measurable (where the codomain is given its product topology and corresponding Borel $\sigma$-algebra).
Then $\bs{\Gr}(\omega)$ can be considered for each $\omega \in \Omega$ as a (possibly empty, possibly not locally finite, possibly disconnected) network in the following way.
For each $u,v \in V$, interpret 
\begin{enumerate}
    \item\label{it:d1} $f_V(v)$ as the indicator that $v \in V(\bs{\Gr})$,
    \item\label{it:d2} $f_E(u,v)$ as the indicator that the edge $\set{u,v} \in E(\bs{\Gr})$,
    \item\label{it:d3} $\xi_V(u)$ as the mark of $u$, and 
    \item\label{it:d4} $\xi_E(u,v)$ as the mark of the vertex-edge pair $(u, \set{u,v})$.
\end{enumerate}
That is, use \cref*{it:d1,it:d2,it:d3,it:d4} to \emph{define} $V(\bs{\Gr}), E(\bs{\Gr})$, and the marks of vertices and edges in $\bs{\Gr}$.
Note that $f_E$ must be symmetric because edges are not directed, but $\xi_E$ may not be, since each edge is associated with two marks, one per vertex.
If one wants to consider directed edges, one instead uses undirected edges and uses the marks on edges to specify which direction the edge should point.
The definition of $\xi_V(u)$ when $u \notin V(\bs{\Gr})$ is irrelevant, 
and similarly for the definition of $f_E(u,v)$ and $\xi_E(u,v)$ if either of $u$ or $v$ is not in $V(\bs{\Gr})$.
All statements about the network defined by $\bs{\Gr}$ are then translated into statements about the maps $(f_V, \xi_V, f_E, \xi_E)$.
For instance, 
\begin{align*}
    \text{\{$\bs{\Gr}$ is not empty\}} =
    \set{\sum_{v \in V} f_V(v) > 0}.
\end{align*}
This is exactly the kind of construction used to define the Doeblin graph $\G$.
In the case of $\G$, 
\begin{enumerate}
    \item $f_V=1$ on $\Z \times S$,
    \item $f_E\left((t,x), (t+1, h(x, \xi_{t}))\right)
        =f_E\left((t+1, h(x, \xi_{t})),(t,x)\right) =1$ for all $(t,x) \in \Z \times S$ 
        and $f_E = 0$ otherwise,
    \item $\xi_V(t,x) = \xi_{t}$ for all $(t,x) \in \Z \times S$, and
    \item $\xi_E\left((t,x), (t+1, h(x, \xi_{t,x}))\right)=1$ for all $t \in \Z, x \in S$ to indicate the edge is directed forwards in time.
\end{enumerate}
This construction also works for the bridge graph $\B$ as well.
When a $\bs{\Gr}$ has been constructed as in this section,
one can see $\bs{\Gr}$ as a random network after any measurable choice of root, given that it is nonempty and locally finite.

\begin{lemma}\label{lem:class-of-tangible-network-is-measurable}
    Suppose $\bs{\Gr}=(f_V,\xi_V,f_E,\xi_E)$ is as above
    and a.s.\ $\bs{\Gr}$ is nonempty, locally finite, and connected.
    Then for any measurable choice of root $\bs{o} \in V(\bs{\Gr})$,
    $[\bs{\Gr}, \bs{o}]$ is a random network.
\end{lemma}

\begin{postponedproof}
    Write the event that $[\bs{\Gr},\bs{o}]$ is within $\epsilon>0$ of some fixed network $[\Gr, o]$
    as a countable union over rooted isomorphic copies $(\Gr', o')$ of $(\Gr,o)$ with vertices in $V$
    of the event that $\bs{o}=o'$, the neighborhood of radius $\lceil\frac{1}{\epsilon}\rceil$ around $\bs{o}$ is exactly $\Gr'$,
    and the marks $\xi_V(u)$ for $u \in V(\Gr')$ and $\xi_E(v,w)$ for $\set{v,w} \in E(\Gr')$ are within $\epsilon$
    of the corresponding vertex and edge marks of $(\Gr', o')$.
    Each of these conditions individually are written in terms of events using the maps $f_V,\xi_V, f_E, \xi_E$,
    showing the desired measurability of $\omega \mapsto [\bs{\Gr}(\omega), \bs{o}(\omega)]$.
    The full proof of \Cref{lem:class-of-tangible-network-is-measurable} given in the appendix.
\end{postponedproof}

Thus indeed $\G$ may be seen as a random network when rooted and marked, assuming it is locally finite and connected.
But the question remains whether this may be done in such a way as to make $\G$ unimodular.
The first approach one might take 
is to investigate whether $\G$, rooted at $(0,X_0)$ for some (random) choice of
$X_0\in S$,
is unimodular.
Two natural choices, at least in the standard CFTP setup,
are to take $X_0$ to be the output of the CFTP algorithm, 
or to take $X_0$ to be independent of $\G$.
For simplicity, the standard CFTP setup refers to the case where $\G$ 
has fully independent transitions, $S$ is finite, and the CFTP algorithm succeeds a.s.
The following proposition determines when $\G$ can be unimodular under the
previous choices of $X_0$.

\begin{proposition}\label{prop:naiive-unimodularity-in-G}
    Suppose $\G$ is an \eft{}, that $\G$ has each $(t,x) \in V(\G)$ marked by $(x, \xi_{t})$, and that $X_0$ is a random choice in $S$.
    Then
    \begin{itemize}
        \item{} if $[\G,(0,X_0)]$ is unimodular, then $S$ is finite and $X_0$ is uniformly distributed on $S$,
        \item{} if $X_0$ is independent of $\G$ and uniformly distributed on a finite $S$, then $[\G, (0,X_0)]$ is unimodular, and
        \item{} if $X_0$ is the output of the CFTP algorithm in the standard CFTP setup, then
            $[\G, (0,X_0)]$ is unimodular if and only if $S$ has a single element.
    \end{itemize}
\end{proposition}

\begin{postponedproof}
    The first point follows by constructing for each $x,y \in S$ a mass transport that, when applied to $\G$,
    sends mass 1 within vertical slices of $\G$ from the vertex in state $x$ to the vertex in state $y$.
    Unimodularity then gives $\P(X_0 = x) = \P(X_0 = y)$.
    The second point follows from the definition of unimodularity.
    The third point follows by noting that the output of the CFTP algorithm
    has at least one child, but unimodularity implies that it must have
    one on average, so a.s.\ it has one child.
    A nonempty tree where every vertex has one incoming and one outgoing edge is isomorphic to $\Z$, so $S$ can only have one state.
    The full proof of~\Cref{prop:naiive-unimodularity-in-G} is given in the
	appendix.
\end{postponedproof}

While choosing $X_0$ uniformly distributed
on $S$ and independent of $\G$ works when $S$ is finite,
unimodularity of the whole $\G$ is doomed in the general case, as there
is no uniform distribution on an infinite $S$.
This is the reason for introducing the bridge graph $\B$,
which is locally finite.
However, the bridge graph may still not be connected, so a spine is added to it
to make it connected.

\begin{corollary}\label{cor:bridge-graphs-embedding}
    Let $\overline{\B}$ be $\B$ \defn{with spine added},
    i.e.\ with edges from each $(t,x^*)$ to $(t+1,x^*)$ for all $t \in \Z$ added.
    Then for any measurable marks and any measurable choice of root $\bs{o} \in V(\B)$, 
    $[\overline{\B}, \bs{o}]$ is a random network.
\end{corollary}

\begin{proof}
    One has that $(0,x^*) \in V(\overline{\B})$, so $\overline{\B}$ is nonempty.
    Also $\overline{\B}$ is locally finite by \Cref{prop:mbeff-is-locally-finite}
    and the fact that adding the spine has increased the degree of each vertex by at most two.
    Finally, since each $v \in V(\overline{\B})$ is connected to some $(t, x^*)$, and the spine in $\overline{\B}$
    connects all such vertices, $\overline{\B}$ is connected.
    \Cref{lem:class-of-tangible-network-is-measurable} finishes the claim.
\end{proof}

Everything is in place to see the unimodular structure hidden in $\G$, which is handled in the next section.

\section{Unimodularizability and its Consequences}\label{sec:main-thm}

\subsection{Unimodularizability of the Bridge Graph}\label{sec:unimodularizability-of-bridge}

The following result identifies the unimodular structure inside $\G$.
For the rest of the document, each $(t,y)\in V(\B)$ is marked by $(y,\xi_{t})$ whenever considered as a vertex in a
rooted network.

\begin{theorem}\label{thm:main-thm-mbeft-is-unimodularizable}
    Any random network with distribution
    \begin{equation}\label{eq:unimodular-measure}
        \Psq(A) := \frac{1}{\E[\#\B_0]} \E\left[ \sum_{w \in V_0(\B)} 1_{\set{[\overline{\B}, w] \in A}}\right],
        \qquad A \in \mathcal{B}(\mathcal{G}_*).
    \end{equation}
    is unimodular. 
    The spine need not be added and $\B$ may also be used instead of $\overline{\B}$ if $\B$ is already connected.
\end{theorem}

One may interpret the distribution $\Psq$ as a size-biased version of
the network obtained by starting with $\B$ and selecting the root uniformly from $\B_0$.

\begin{proof}
    By \Cref{cor:bridge-graphs-embedding}, $\overline{\B}$ with marks as specified and any choice of root
    is a random network.
    Therefore, all the quantities in the following calculation are measurable.
    Let $g:\mathcal{G}_{**} \to \R_{\geq 0}$ be given.
    One has
	\begin{align*}
        &\int_{\mathcal{G}_{*}} \sum_{v \in V(\Gr)}
        g[\Gr,o,v]\,\Psq(d[\Gr,o])\\
        &= \frac{1}{\E[\#\B_0]} \E \sum_{y \in
            \B_0}\sum_{v \in V(\B)} g[\overline{\B},(0,y), v]\\
        &= \frac{1}{\E[\#\B_0]} \sum_{y,y' \in S, t \in \Z}\E \left[
            1_{\set{(0,y), (t,y') \in V(\B)}}g[\overline{\B},(0,y),
            (t,y')]\right].
    \end{align*}
    Stationarity on $\Z$ implies the right hand side is equal to
    \begin{align*}
        & \frac{1}{\E[\#\B_0]} \sum_{y,y' \in S, t \in \Z}\E \left[
            1_{\set{(-t,y), (0,y') \in V(\B)}}g[\overline{\B},(-t,y),
			(0,y')]\right]\\
        = &\frac{1}{\E[\#\B_0]} \sum_{y,y' \in S, t \in \Z}\E \left[
            1_{\set{(t,y), (0,y') \in V(\B)}}g[\overline{\B},(t,y),
			(0,y')]\right]\\
        = &\frac{1}{\E[\#\B_0]} \E \sum_{y' \in
            \B_0}\sum_{v \in V(\B)} g[\overline{\B},v, (0,y')]\\
      =&\int_{\mathcal{G}_{*}} \sum_{v \in V(\Gr)}
      g[\Gr,v,o]\,\Psq(d[\Gr,o]).
	\end{align*}
    Thus $\Psq$ is the distribution of a unimodular network.
\end{proof}

The view of $\Psq$ as a size-biased version of a network is formalized in the following.

\begin{proposition}\label{prop:size-biased-dist}
    Let $\bs{o}$ be, conditionally on $V_0(\B)$,
    uniformly distributed on $V_0(\B)$ and independent of $\B$.
    Then under the size-biased measure $\hat \P(E) := \frac{1}{\E[\#\B_0]} \E[\#\B_0 1_E]$ for each $E \in \F$,
    the random network $[\overline{\B}, \bs{o}]$ has the distribution $\Psq$.
\end{proposition}

\begin{proof}
    In what follows, $V$ ranges over the sets for which $\P(V_0(\B) = V) > 0$, of which
    there are at most countably many because $\B_0$ is a.s.\ a finite subset of the countable $S$.
    For any $A \in \mathcal{B}(\mathcal{G}_*)$ and with $C := \E[\#\B_0]$,
    \begin{align*}
        &\hat \P([\overline{\B}, \bs{o}] \in A) \\
        &= \frac{1}{C} \E[\# \B_0 1_{\set{[\overline{\B}, \bs{o}] \in A}}]\\
        &= \frac{1}{C}\sum_{V} \left\vert V\right\vert \P(V_0(\B)=V) \P([\overline{\B}, \bs{o}] \in A \mid V_0(\B) = V)\\
        &= \frac{1}{C}\sum_{V} \sum_{v\in V} \left\vert V\right\vert \P(V_0(\B)=V) 
        \P(\bs{o} = v, [\overline{\B}, v] \in A \mid V_0(\B) = V)
    \end{align*}
    which, by the conditional independence of $\bs{o}$ and $\B$, is
    \begin{align*}
        &= \frac{1}{C}\sum_{V} \sum_{v \in V}\left\vert V\right\vert \P(V_0(\B)=V) 
        \frac{1}{\left\vert V\right\vert}\P([\overline{\B}, v] \in A \mid V_0(\B) = V)\\
        &= \frac{1}{C}\E\left[\sum_{V} \sum_{v \in V} 
            1_{\set{[\overline{\B}, v] \in A, V_0(\B) = V}}\right]\\
        &= \Psq(A)
    \end{align*}
    as claimed.
\end{proof}

\subsection{I/F Component Properties}\label{sec:i-f-component-properties}
For any measurable event $A\subseteq \mathcal{G}_*$ in the $\sigma$-algebra of 
\defn{root-invariant} events, i.e., such that if $[\Gr,o] \in A$ then $[\Gr,v] \in A$ for
all $v \in V(\Gr)$, one has
\begin{align*}
    \Psq(A) 
    = \frac{1}{\E[\#\B_0]}\E\left[\sum_{w \in
            V_0(\B)} 1_{\set{[\overline{\B}, w] \in A}} \right]
    = \frac{1}{\E[\#\B_0]}\E\left[\#\B_0
        1_{\set{[\overline{\B} , (0,x^*)] \in A}} \right].
\end{align*}
This immediately gives the following.

\begin{lemma}\label{lem:same-root-inv-events-measure-zero}
    One has that $\Psq$ and $\P([\overline{\B},(0,x^*)] \in \cdot)$
    have the same root-invariant sets of measure $0$ or $1$. \qed{}
\end{lemma}

Next, a vertex-shift that is designed to follow the arrows in $\B$ is defined.
It plays the same role as $\fnext$ but is defined for all networks.
From now on, let $\Phi$ denote the \defn{follow vertex-shift} defined on any network $\Gr$ for each $u \in V(\Gr)$
by $\Phi_{\Gr}(u):=v$ if either:
\begin{enumerate}
    \item there is a unique outgoing edge from $u$ and this edge terminates at $v$, or
    \item $u$ is in state $x^*$ and there is a unique outgoing edge from $u$
        that does not terminate at a vertex in state $x^*$, and this edge terminates at $v$.
\end{enumerate}
If neither of the two conditions above is met for any $v \in V(\Gr)$, define $\Phi_{\Gr}(u):=u$ for concreteness.
Here a vertex is considered to be in a state $y \in S$ when the first
component of its mark is $y$ (recall that a vertex $(t,y) \in V(\B)$ is marked by $(y, \xi_t)$).
The second clause in the definition of $\Phi$ is there because of the presence
of the spine in $\overline{\B}$, so that if the root is in state $x^*$
the vertex-shift will choose to follow the arrow in $\B$ instead of following the arrow to the next element of the spine,
unless the two coincide.
By construction, $\Phi_{\overline{\B}}(t,x) = \fnext(t,x)$ for all $(t,x) \in V(\overline{\B})$.

The event that all $\Phi$-components of a network are of class I/F
is root-invariant, and moreover it has $\P([\overline{\B}, (0,x^*)] \in \cdot)$-probability one
because the $\Phi$-graph of $\overline{\B}$ is $\B$ itself,
the $\Phi$-components of $\overline{\B}$ are the components of $\B$, and the $\Phi$-foils
of $\overline{\B}$ are subsets of the sets $\Fam{V_t(\B)}_{t \in \Z}$, which are finite.
Hence $\Psq$ is concentrated on the set of networks having only $\Phi$-components of I/F class.
It follows that any a.s.\ root-invariant properties that follow from $\Psq$ being unimodular and having I/F components
automatically apply to $\P([\overline{\B}, (0,x^*)] \in \cdot)$ as well.
Such properties will be referred to as \defn{I/F component properties} and are
explored in \Cref{sec:bi-recurrence,sec:other-i-f-properties}.

\subsubsection{Bi-recurrent Paths}\label{sec:bi-recurrence}

This section studies bi-infinite paths in $\G$ and identifies special bi-infinite paths
that have a certain recurrence property backwards in time.
Firstly, it is possible to have multiple bi-infinite paths in $\G$ because $\G$ is disconnected.
\begin{example}
    Consider the case where $S := \set{1,2}$
    and $\hgen$ and $\Fam{\xi_{t}}_{t \in \Z}$ are chosen so that the transition $(t,1) \to (t+1,2)$
    occurs if and only if $(t,2) \to (t+1,1)$ occurs.
    In this case $\G$ has two components a.s.
    Each component is itself a bi-infinite path.
\end{example}

Moreover, even when $\G$ is connected, it it still possible to have multiple bi-infinite paths in $\G$.

\begin{example}\label{ex:deterministic-falling}
    Consider the case of $S:=\N$ with fully independent transitions.
    Let the transition matrix $P$ be determined as follows.
    In state $0$, transition to a $\mathrm{Geom}(1/2)$ random variable, 
    and from any other $n \neq 0$, deterministically transition from $n$ to $n-1$.
    In this case, from every vertex $(s,x) \in \Z\times S$, 
    there is a bi-infinite path $\Fam{t,X_t}_{t \in \Z}$ in $\G$
    for which $X_{s-k} = k+x$ for all $k \geq 0$.
    Thus there are infinitely many bi-infinite paths, despite the fact that in this case
    $\G$ is an \eft{}, which follows from~\Cref{prop:teaser-indep-transitions-components}.
\end{example}

In~\Cref{ex:deterministic-falling}, even though $\G$ is connected, $\G$ has infinitely many bi-infinite paths.
However, amongst the bi-infinite paths,
there is one special bi-infinite path.
The special path is the unique bi-infinite path that visits every state infinitely often
\emph{in the past}.
It turns out that this is the correct kind of path to look for in general.
A bi-infinite sequence $\Fam{x_t}_{t \in \Z}$ in $S$ is called \defn{bi-recurrent for state $x$} 
if $\set{t\in \Z : x_t = x}$ is unbounded above and below.
If $\Fam{x_t}_{t \in \Z}$ is bi-recurrent for every $x\in S$, it is simply called \defn{bi-recurrent}.
A state path $\Fam{X_t}_{t \in I}$ in $\G$ is called \defn{bi-reccurent (for state $x$)}
if a.s.\ its trajectory is bi-recurrent (for state $x$).
Recall that $\Phi$ denotes the follow vertex-shift.
The existence of bi-infinite paths in $\Phi$-components of a network
is an I/F property, and hence one has the following.

\begin{proposition}\label{prop:existence-of-bi-rec-paths}
    It holds that $\B$ has a unique bi-infinite path in each component a.s.
    The corresponding state paths are bi-recurrent for $x^*$ and these are the only state paths
    in all of $\G$ that are bi-recurrent for $x^*$.
    Moreover, for each $y \in S$, these state paths either a.s.\ never visit $y$, or are bi-recurrent for $y$.
\end{proposition}

\begin{proof}
    By \Cref{thm:foil-classification-theorem},
    $\Psq$-a.e.\ network has a unique bi-infinite path in each $\Phi$-component, where $\Phi$ is the follow vertex-shift.
    But having a unique bi-infinite path in each $\Phi$-component is a root-invariant event,
    and hence $\P$-a.s.\ $\overline{\B}$ has a unique bi-infinite path in each $\Phi$-component.
    Since the $\Phi$-components of $\overline{\B}$ are the components of $\B$,
    $\P$-a.s.\ every component of $\B$ contains a unique bi-infinite path.

    Let $\Pi$ be the covariant partition of $\Phi$-components.
    Define the covariant subset $C$ on a network $\Gr$ by letting $C_\Gr$ be the subset of vertices of $\Gr$
    that are either the first or last visit to a given state $y \in S$, if they exist, on the unique bi-infinite path in their $\Phi$-component of $\Gr$, if such a path exists.
    The no infinite/finite inclusion lemma, \Cref{lem:no-infinite-finite-inclusion},
    implies that $\Psq$ is concentrated on the set of networks $\Gr$
    with no first or last visit to $y$ on the unique bi-infinite paths in each $\Phi$-component of $\Gr$.
    This property is root-invariant and hence a.s.\ the state paths corresponding to the unique bi-infinite paths in each component of $\B$
    either do not visit state $y$ or are bi-recurrent for $y$.
    Taking a countable union over $y \in S$ shows this property holds simultaneously for all $y \in S$.
    Since the unique bi-infinite path in each component of $\B$ at least hits $x^*$, 
    one may at least conclude the paths are bi-recurrent for $x^*$.
    Finally, there cannot be any other bi-recurrent state paths for $x^*$ in $\G$ because, by definition,
    a bi-recurrent state path in $\G$ will lie in $\B$ since it visits $x^*$
    at arbitrarily large negative times.
\end{proof}

The next result applies~\Cref{prop:existence-of-bi-rec-paths} to the nicest case,
where $\G$ is a tree. 

\begin{corollary}\label{cor:ergodic-meft-unique-bi-recurrent}
    Suppose that $\G$ is an \eft{}.
    Then $\G$ contains a unique (up to measure zero modifications) state path $\Fam{\beta_t}_{t \in \Z}$ that is bi-recurrent for $x^*$.
    Moreover, $\Fam{\beta_t}_{t \in \Z}$ is shift-covariant, stationary, and
    for each $t \in \Z$ one has that $\beta_t$ is measurable with respect to $\sigma(\xi_s : s < t)$.
    Additionally, $\Fam{\beta_t}_{t \in \Z}$ is bi-recurrent for every $x \in S$ that is positive recurrent.
\end{corollary}

\begin{proof}
    \Cref{prop:existence-of-bi-rec-paths} 
    shows that a.s.\ there is a unique bi-infinite path 
    in each component of $\B$, and the corresponding state paths are bi-recurrent for $x^*$.
    Since $\G$ a.s.\ has only one component, $\B$ does too.
    The second part of~\Cref{prop:existence-of-bi-rec-paths} then implies
    the bi-recurrent state path for $x^*$ in $\B$ is the only bi-recurrent state path for $x^*$ in $\G$.
    One would like to define $\Fam{\beta_t}_{t \in \Z}$ to be the unique bi-recurrent state path for $x^*$ in $\G$.
    However, in that case, $\Fam{\beta_t}_{t \in \Z}$ would only be defined a.s.
    For concreteness,
    define $\beta_t$ for each $t \in \Z$ by letting 
    $\beta_t := \lim_{s \to -\infty} F^{(s,x^*)}_t$ on the event that the limit exists, and $\beta_t := x^*$ otherwise.
    On the a.s.\ event $E$ that $\B$ is connected, $\#\B_t < \infty$ for all $t \in \Z$, and there is a unique bi-infinite path in $\B$, 
    one has that $\Fam{t,\beta_t}_{t \in \Z}$ coincides with the unique bi-infinite path in $\B$. 
    This is because if, for some $t \in \Z$, $\lim_{s \to -\infty} F_t^{(s,x^*)}$ does not exist, then
    either $\#\B_t = \infty$, or there exist two states $x,y \in S$ such that $(t,x)$ and $(t,y)$
    have (necessarily disjoint) locally finite infinite trees of descendants in $\B$.
    The former case is forbidden on $E$, and, in the latter case,
    K\"onig's lemma would imply the existence of two distinct bi-infinite paths in $\B$, which is also forbidden on $E$.
    Thus $\lim_{s \to -\infty} F_t^{(s,x^*)}$ exists for all $t \in \Z$ on the event $E$,
    and, on this event, the unique bi-infinite path in $\B$
    must therefore be $\Fam{t,\beta_t}_{t \in \Z}$.
    The shift-covariance and hence stationarity of $\Fam{\beta_t}_{t \in \Z}$ follows from
    its definition in terms of $F^{(s,x^*)}$ for each $s \in \Z$.
    For each $t \in \Z$, measurability of $\beta_t$ with respect to $\sigma(\xi_s : s < t)$
    also follows from its definition, since each $F^{(r,x^*)}_t$ with $r \leq t$ is $\sigma(\xi_s : s < t)$-measurable.

    Now let $\Fam{Y_t}_{t \in \Z}$ be the unique bi-recurrent state path for some other $y \in S$ that is positive recurrent.
    Since $\G$ is a.s.\ connected, $\Fam{\beta_t}_{t \in \Z}$ and $\Fam{Y_t}_{t \in \Z}$ eventually merge, a.s.
    However, stationarity forbids that there is a first time such that $\beta_t = Y_t$, so it must be that $\beta_t = Y_t$ for all $t \in \Z$.
    Thus $\Fam{\beta_t}_{t \in \Z}$ is bi-recurrent for every $y \in S$ that is positive recurrent.
\end{proof}

\Cref{cor:ergodic-meft-unique-bi-recurrent} shows that, like in the standard CFTP setup,
there is a $\beta_0$ living at time $0$ in $\G$ that is a perfect sample from the stationary distribution of the Markov chain or SRS\@.
However, unlike in the standard CFTP setup, it is not known whether there is an algorithm that can find $\beta_0$
in finite time.

Another consequence of the existence of bi-recurrent paths in $\B$ is that one
can bound the number of components of $\B$.

\begin{corollary}\label{cor:B-has-finitely-many-components}
    The a.s.\ constant number $n$ 
    of components of $\B$ 
    is no larger than $\min\set{ k : \P(\#\B_0 = k)>0} < \infty$.
    In particular, $\B$ has finitely many connected components,
    even if $\G$ has infinitely many components, and if $\P(\#\B_0 = 1) > 0$, 
    then $\B$ is an \eft{}.
\end{corollary}

\begin{proof}
    The number of components of $\B$ is shift-invariant and hence a.s.\ constant.
    Each component of $\B$ contains a bi-recurrent path by \Cref{prop:existence-of-bi-rec-paths}.
    Each bi-recurrent path intersects $V_0(\B)$ in a different element since they are in different components of $\B$.
    It follows that $n \leq \#\B_0$ a.s.
    If $\P(\# \B_0 = k) > 0$ for some $k$, then it follows that $n \leq k$.
\end{proof}

The deterministic cycle on $n$ states shows that the bound in \Cref{cor:B-has-finitely-many-components}
can be achieved for each $n$.
In general, any bi-infinite stationary process on $S$ (or any countable set) must be bi-recurrent.

\begin{proposition}\label{prop:stationary-implies-birec}
    Suppose that $\Fam{X_t}_{t \in \Z}$ is a stationary process taking values in $S$.
    Then a.s.\ $\Fam{X_t}_{t \in \Z}$ is bi-recurrent for every $x \in \set{X_t}_{t \in \Z}$.
\end{proposition}

\begin{proof}
    For each $x \in S$, 
    stationarity forbids that there is a first or last visit of $\Fam{X_t}_{t \in \Z}$ to $x$
    since such an occurrence would have to be equally likely to happen at all times $t \in \Z$.
    Thus, a.s.\ either $x \notin \set{X_t}_{t \in \Z}$ or $\set{t \in \Z: X_t = x}$ must be 
    unbounded both above and below.
    The countability of $S$ finishes the claim.
\end{proof}

The remainder of the section specializes to the Markovian setting again.
In the Markovian setting, bi-recurrence is actually equivalent to
stationarity in the irreducible, aperiodic, positive recurrent case.

\begin{theorem}\label{thm:stationary-iff-birecurrent}
    Suppose that $P$ is irreducible, aperiodic, and positive recurrent,
    and that $\Fam{X_t}_{t \in \Z}$ is a Markov chain with transition matrix $P$.
    Then $\Fam{X_t}_{t \in \Z}$ is stationary if and only if it is bi-recurrent for any (and hence every) state.
\end{theorem}

\begin{proof}
    By~\Cref{thm:embed-into-general-meff}, it is possible to assume without loss of generality
    that $\Fam{X_t}_{t \in \Z}$ is a state path in the Doeblin graph $\G$ with fully independent transitions.
    By~\Cref{prop:teaser-indep-transitions-components}, $\G$ is an \eft{} and therefore~\Cref{cor:ergodic-meft-unique-bi-recurrent}
    implies that $\G$ contains a bi-recurrent state path $\Fam{\beta_t}_{t \in \Z}$ that is, for all $y \in S$, 
    the a.s.\ unique bi-recurrent state path for state $y$ in $\G$.
    Moreover, $\beta_t \sim \pi$ for all $t \in \Z$, where $\pi$ is the stationary distribution for $P$.
    If $\Fam{X_t}_{t \in \Z}$ is bi-recurrent for some $y\in S$, then, by uniqueness, $X_t = \beta_t$ for all $t \in \Z$, a.s.
    In particular, $\Fam{X_t}_{t \in \Z}$ is stationary.
    The converse follows from \Cref{prop:stationary-implies-birec} and irreducibility.
\end{proof}

A bi-infinite path in $\G$ whose state path is not bi-recurrent
for any state $x \in S$ will be called \defn{spurious}.
Observe the difference between spurious bi-infinite paths and the unique bi-recurrent path
in \Cref{fig:meft-falling-example}.
Viewed in reverse time, a spurious path must run off to $\infty$
in the sense that for every finite set $F\subseteq S$, the reversed path eventually leaves $F$ forever.
It is possible for $\G$ to contain spurious bi-infinite paths, as was seen
in~\Cref{ex:deterministic-falling}.

\begin{figure}[t!]
    \centering
    \includegraphics[width=\linewidth]{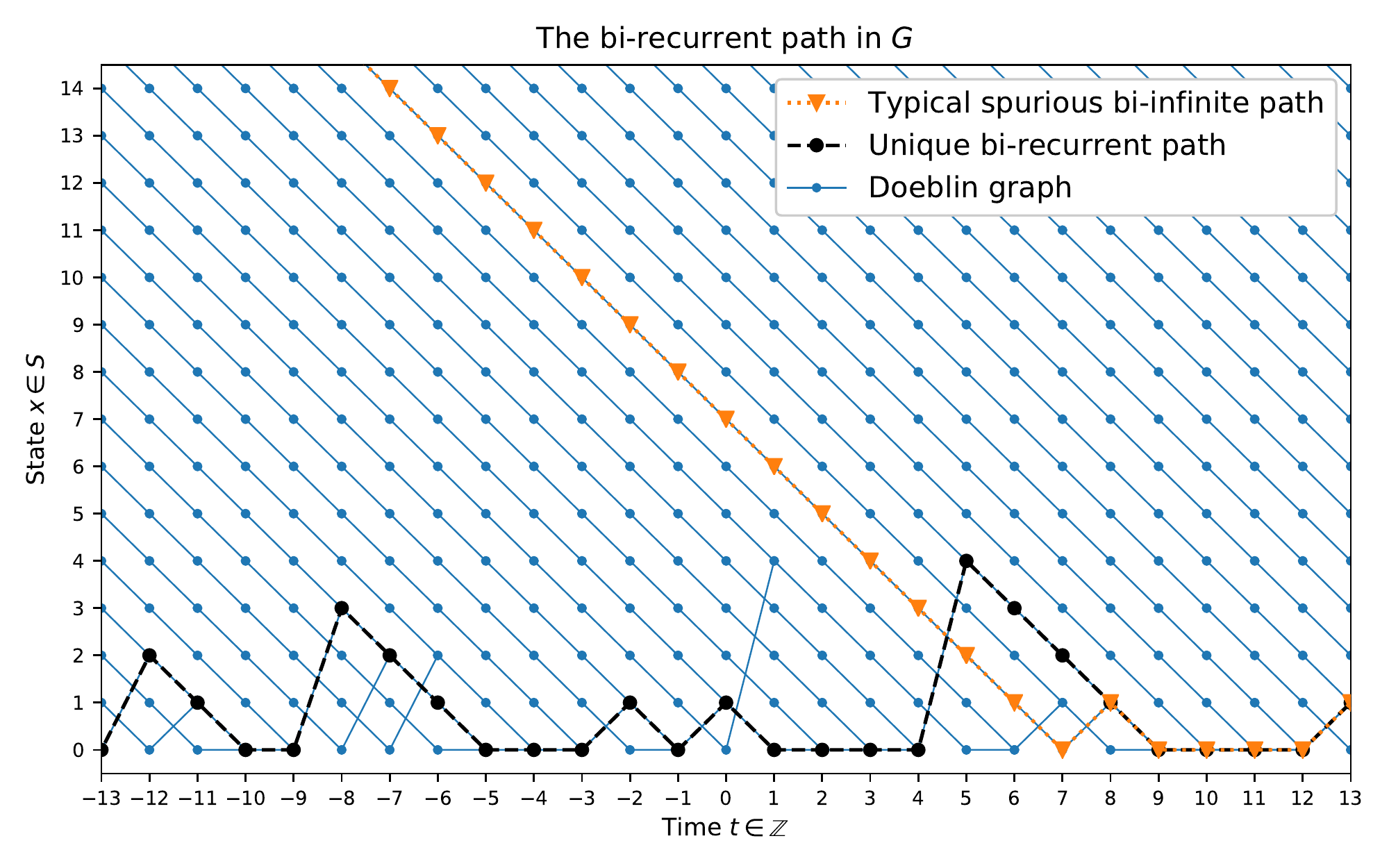}
    \caption{The Doeblin graph from~\Cref{ex:deterministic-falling} with the bi-recurrent
        path and a spurious path distinguished.%
        \label{fig:meft-falling-example}}
\end{figure}

Say that $P^n$ \defn{converges uniformly (to $\pi$ as $n \to \infty$)} if $P$ is irreducible, aperiodic, and positive recurrent
with stationary distribution $\pi$, and
$\sup_{x \in S}\left\Vert P^n(x,\cdot) - \pi\right\Vert \to 0$ as $n\to \infty$.
For example, this is automatic if $P$ is irreducible, aperiodic, and $S$ is finite.
Some authors call $P$ uniformly ergodic, but the term ergodic is not used here
to avoid a terminology collision with ergodic theory.
Uniform convergence to $\pi$ is also equivalent (cf.~\cite{meyn2009markov} Theorem 16.0.2 (v))
to the statement that there is $m$ such that
$P^m(x,\cdot) \geq \varphi(\cdot)$ for all $x \in S$, for a measure $\varphi$ which is not the zero measure.
It is also equivalent (cf.~\cite{foss1998perfect} Theorem 4.2) 
to the fact that the CFTP
algorithm succeeds in the case of fully independent transitions, i.e.\ the backwards vertical coupling time 
$\inf \set{t>0 : F^{(-t,x)}_{0} = F^{(-t,y)}_{0}, \forall x,y \in S}$ is 
a.s.\ finite.

Together, the following two results say that when a Markov chain that mixes uniformly is started in the infinite past,
it has converged to its stationary distribution by any finite time.

\begin{proposition}\label{prop:unif-ergodicity-no-spurious-paths}
    Suppose $P^n$ converges uniformly to $\pi$ as $n \to \infty$ and $\G$ has fully independent transitions.
    Then $\G$ contains no spurious bi-infinite paths.
\end{proposition}

\begin{proof}
    For every $s<t$ let $C_{s,t}$ be the event that $F^{(s,x)}_t = F^{(s,y)}_t$ for all $x,y \in S$.
    That is, $C_{s,t}$ is the event that starting at time $s$, all paths in $\G$ collapse to a single state by time $t$.
    Note that $\P(C_{s,t})$ depends only on $t-s$.
    Since $\G$ has fully independent transitions and $P^n$ converges to $\pi$ uniformly as $n \to \infty$, by e.g.\ Theorem 5.2 in~\cite{foss1998perfect},
    there exists some $k \in \N$ such that
    $\P(C_{s,t}) > 0$ when $t-s\geq k$.
    Consider $E_n := C_{-k(n+1), -kn}$ for each $n \in \N$.
    One has $\P(E_n) = \P(E_0)>0$ for all $n$ and the $E_n$ are independent.
    It follows that a.s.\ infinitely many of them occur.
    On an $\omega$ for which infinitely many $E_n$ occur,
    there is at most one bi-infinite path in $\G$, and thus 
    any bi-infinite path in $\G$ must coincide with the unique bi-recurrent path guaranteed to
    exist by~\Cref{cor:ergodic-meft-unique-bi-recurrent}.
\end{proof}

It is a classical result that it is possible to find a bi-infinite stationary version
$\Fam{X_t}_{t \in \Z}$
of a Markov chain that has a stationary distribution.
The following shows that, in the case of uniform convergence to $\pi$,
this is the only way to extend a Markov chain to have
time index set all of $\Z$.
That is, if $\Fam{X_t}_{t \in \Z}$ is a Markov chain that conveges uniformly
to its stationary distribution, then it must be that $X_t
\sim \pi$ for all $t \in \Z$.

\begin{proposition}\label{prop:bi-infinite-stationary-unif-ergodic}
    Suppose $P^n$ converges uniformly to $\pi$ as $n \to \infty$.
    Then every Markov chain $\Fam{X_t}_{t \in \Z}$ with transition matrix $P$
    is stationary and bi-recurrent.
    The subtle assumption here is that the time index set is all of $\Z$.
\end{proposition}

\begin{proof}
    By \Cref{thm:embed-into-general-meff}, one may assume $\Fam{X_t}_{t \in \Z}$
    is a state path in $\G$ with fully independent transitions,
    which is then an \eft{} by \Cref{prop:teaser-indep-transitions-components}.
    Since $P^n$ converges uniformly to $\pi$ as $n \to \infty$, $\G$ contains
    no spurious bi-infinite paths by \Cref{prop:unif-ergodicity-no-spurious-paths},
    and hence $\Fam{X_t}_{t \in \Z}$ must be the bi-recurrent state path.
    \Cref{thm:stationary-iff-birecurrent} then implies $\Fam{X_t}_{t \in \Z}$
    is stationary.
\end{proof}

\Cref{prop:bi-infinite-stationary-unif-ergodic} 
may fail for an irreducible, aperiodic, and positive recurrent $P$
if $P$ does not converge uniformly to its stationary distribution.
Indeed, it was already shown, e.g., in~\Cref{ex:deterministic-falling},
that it is possible for $\G$ to admit spurious bi-infinite paths.
If $\Fam{X_t}_{t \in \Z}$ is a proper state path in $\G$ that corresponds to a spurious bi-infinite path, then
$\Fam{X_t}_{t \in \Z}$ is a Markov chain with transition matrix $P$, but it is not stationary since it is not bi-recurrent.
Recall that $\B(x)$ denotes the bridge graph in $\G$ using $x$ as the base point instead of $x^*$.

\begin{proposition}\label{prop:intersection-is-bi-inf-path}
    Suppose $P$ is irreducible, aperiodic, and positive recurrent, and that $\G$
    has fully independent transitions.
    If
    \begin{enumerate}
        \item{} $S$ is infinite, 
        \item{} $\G$ is locally finite, and 
        \item{} $\G$ contains no spurious bi-infinite paths,
    \end{enumerate}
    then
    \begin{align}
        \bigcap_{x \in S} V(\B(x)) = \set{(t,\beta_t): t \in \Z},
    \end{align} 
    where $\Fam{\beta_t}_{t \in \Z}$ is the unique bi-recurrent state path in $\G$.
    That is, the bi-recurrent path in $\G$ is the only thing common to all of the \befts{}.
    Alternatively,
    if $S$ is finite and has at least 2 states, then a.s.\ 
    \begin{align}
        \bigcap_{x \in S} V(\B(x)) \supsetneq \set{(t,\beta_t): t \in \Z}.
    \end{align}
\end{proposition}

\begin{proof}
    For each $x \in S$, the bi-recurrent path is in $\B(x)$ because it is bi-recurrent for $x$.
    Suppose $S$ is infinite, $\G$ is locally finite, and that $\G$ contains no spurious bi-infinite paths.
    Consider a vertex $v \in V(\G)$ not on the bi-recurrent path.
    The tree of all descendants of $v$ in $\G$ must be finite, else K\"onig's
    lemma would give a bi-infinite path in $\G$ that is distinct from the unique bi-recurrent path 
    since $v$ is not on the bi-recurrent path.
    Since $\G$ contains no spurious bi-infinite paths, this is impossible.
    Since the tree of descendants of $v$ is finite but $S$ is infinite, there is
    some state $x \in S$ such that $v$ has no descendant in state $x$.
    In particular, $v \notin V(\B(x))$, showing that nothing off the bi-recurrent path
    can be common to all the \befts{}.
    
    Next suppose that $2 \leq \# S < \infty$.
    It suffices to give a finite deterministic graph $\Gr$ that is a subgraph of $\G$ with positive probability
    such that when some time-translate of $\Gr$ is a subgraph of $\G$, $\bigcap_{x \in S} V(\B(x))$ 
    contains a vertex not on the unique bi-infinite path in $\G$.
    Firstly, since $S$ is finite, choose a tree $T$ on $\Z \times S$ that occurs with
    positive probability and is an example witnesses of the a.s.\ finiteness of the backwards vertical coupling time
    $\inf \set{t>0 : F^{(-t,x)}_{0} = F^{(-t,y)}_{0}, \forall x,y \in S}$.
    Suppose $T$ is rooted at $(0,x_0)$. In particular, $V(T) \subseteq (-\infty,0] \times S$.
    By irreducibility of $P$ and the fact that $\#S \geq 2$,
    choose $L=(x_0,x_1,\ldots,x_n)$ a finite path in $S$ using only positive probability transitions from $x_0$
    back to $x_0=x_n$ that passes through all states of $S$ and has the property that $x_i \neq x_{i+1}$ for any $i$.
    Note that
    \begin{align}
        L_0 := \set{(t, x_t)  : t=0,\ldots, n},\qquad
        L_1 := \set{(t+1, x_t) : t=0,\ldots, n}
    \end{align}
    do not intersect.
    Moreover, $L_0$ and $T$ intersect only at the vertex $(0,x_0)$, and $L_1$ and $T$ do not intersect.
    Let $\Gr$ be the union of $T$, $L_0$, and $L_1$.
    The edges of $T$, $L_0$, and $L_1$ all occur with positive probability in $\G$, 
    and none of them have the same initial vertex, so that in fact they are comprised of independent edges in $\G$.
    Since $\Gr$ has only a finite number of edges, it follows that $\Gr\subseteq \G$ occurs
    with positive probability.
    Moreover, when $\Gr\subseteq \G$ occurs, the vertex $(n,x_{n-1}) \in V(\B(x))$ for all $x \in S$,
    but it is not on the bi-infinite path.
    This is because, by construction, $(0,x_0)$ is on the bi-infinite path in $\G$ and therefore
    $L_0$ makes up a segment of the bi-infinite path in $\G$.
    But, $L_1$ includes a representative for every state,
    so for every $x \in S$ there is an $s \in \Z$ such that $x \in
    D^{(n,x_{n-1})}_s$.
    Finally, $V(L_0) \cap V(L_1) = \emptyset$ so $(n,x_{n-1})$
    is not on the bi-infinite path in $\G$.
\end{proof}

\subsubsection{Other I/F Component Properties}\label{sec:other-i-f-properties}

The existence and uniqueness of a bi-infinite path in each $\Phi$-component of a network is one I/F
property that was studied at length in \Cref{sec:bi-recurrence}, which centered around bi-recurrent paths in $\B$. 
However, there are many other potential things to say about $\B$ following from its I/F structure.
A few of them are discussed in this brief section.

The first is the general structure of a network with only I/F components.
Each component of $\B$ contains a unique bi-infinite path.
Points on a bi-infinite path are sometimes referred to as \defn{immortals} due to
the fact that they do not disappear after an infinite number of applications of the follow vertex-shift $\Phi$.
A component \defn{evaporates} if each point disappears after a finite number (depending on the point) of applications of $\Phi$.
Thus, in the case of $\B$, none of the components evaporate.
\defn{Mortals} are those points in $V(\B)$ that do disappear after a finite number of applications of $\Phi$, i.e.\ those that have only
finitely many descendants.
Each component of $\B$ contains a bi-infinite path of immortals, and each immortal has exactly one child who is immortal.
Thus the immortals within a component are ordered like $\Z$ in a shift-covariant way.
Hanging off of each immortal is then a (possibly empty) tree of mortals, the
descendants of the immortal who are not themselves immortal and whose closest
immortal ancestor is the given immortal.
With this viewpoint, each component of $\B$ can be seen as a shift-covariant bi-infinite sequence of finite rooted trees, where each immortal
is the root of its tree.
If there is only one component of $\B$, then it has already been noted that there is a
unique bi-infinite path in $\B$ whose state path $\Fam{\beta_t}_{t \in \Z}$ is stationary.
However, more can be said in this case.
If there is only one component of $\B$, then in fact the whole sequence $\Fam{[Q_t, (t,\beta_t)]}_{t \in \Z}$ is stationary,
where $Q_t$ is the tree hanging from the immortal $(t,\beta_t)$.
It is important here that the isomorphism class of $Q_t$ is used and each vertex $(t,y) \in V(\B)$ is marked with $(y,\xi_t)$,
otherwise the sequence would not be stationary due to the strictly increasing time coordinate.
This view of $\B$ as a joining of trees gives an alternative way of looking at $\B$ compared to the view of $\B$
as a union of bridges between $x^*$ at different times.
Yet another viewpoint is that of $\B$ as a sequence of vertical slices.
This idea has already been explored slightly in that the way the root was chosen in the definition of the unimodular measure $\Psq$
is by choosing a root from one of these vertical slices.
The view of $\B$ as a sequence of vertical slices is explored more in \Cref{sec:local-weak-convergence} and
is the main topic of \Cref{sec:renewal-structure-of-mbefts}.

Additionally, the list of mass transports given in the appendix gives some integrability results
relating these three viewpoints. In particular, in each way of viewing $\B$ there is a natural way to split $\B$
into pieces. 
In the view of $\B$ as a joining of a sequence of trees of mortals hanging off an immortal, the vertices are partitioned by which tree they are in.
In the view of $\B$ as a sequence of vertical slices, the vertices are partitioned by which slice they are in.
In the view of $\B$ as paths started from state $x^*$, vertices are partitioned by the time they first return to $x^*$.
In fact, the mass transport arguments given in the appendix show that the mean number of vertices in a partition element
is the same for all three viewpoints.
See the list of mass transports in the appendix for a more detailed description of these results and other finer-grained results.

\subsection{Applications to Simulating the Bridge Graph}\label{sec:simulation-applications}

\subsubsection{Local Weak Convergence to the Bridge Graph}\label{sec:local-weak-convergence}

It was shown in \Cref{prop:size-biased-dist} that the measure $\Psq$ may be thought of
as an appropriately size-biased version of a network with the root picked uniformly at random from individuals at time $0$.
A common reason for size-biasing to show up is when picking uniformly at random
across a population and asking the size of the group an individual is in.
Picking uniformly at random is what unimodularity models, so one might expect
that a unimodular network can be approximated by picking the root uniformly at random from
a very large but finite sub-network.
At present, whether all unimodular networks can be approximated in this way is an 
open problem~\cite{aldous2007processes}.
In the case of the unimodular \beff{}, it will be shown directly that indeed it can be approximated by finite sub-networks
with a root picked uniformly at random.

In this section, different ways of approximating the unimodular version of $\B$ by finite subgraphs
are considered.
Recall that $\overline{\B}$ denotes $\B$ with spine added, i.e.\ $\B$ with edges connecting each $(t,x^*)$ to $(t+1,x^*)$.
For a finite interval $I \subseteq \Z$
define $V_{I}(\B) := \cup_{t \in I} V_t(\B)$ and let $\overline{\B} \cap I$ denote the subgraph of $\overline{\B}$ induced by $V_{I}(\B)$.
Also define $V'_{I}(\B) := \cup_{t \in I} \set{(s,F^{(t,x^*)}_s): t \leq s \leq \sup I}$ to be the vertices of $\B$ obtained by simulating
paths starting from $x^*$ within the time window $I$, and let $\overline{\B} \sqcap I$ denote the graph it induces in $\overline{\B}$.
Two ways of approximating $\B$ are then as follows:
\begin{enumerate}
    \item Restrict to $[-n,0]$ and pick a uniform root in $V_{[-n,0]}(\B)$.
    \item Simulate paths starting from $x^*$ in the window $[-n, n]$, 
        which gives the vertices of $\overline{\B} \sqcap [-n,n] \subseteq \overline{\B}$,
        then pick a uniform root in $V'_{[0,n]}(\B)$.
\end{enumerate}
After choosing a large viewing window $I$, 
a vertex picked at random will not likely be near the edge of this window, so the effects of
throwing away all but this finite window can be controlled.
However, the first method involves perfect knowledge of some finite window of $\B$.
Practically speaking, when $S$ is infinite, one does not have a way to be sure that one has computed all of $\B$ in a finite window,
as the only tool available is to simulate sample paths starting from different locations.
This is the motivation for the second method of picking a root.
For, even if the edge effects caused by only viewing simulations of paths in $\B$ from
$-n$ to $n$ cannot be controlled, the edge effects from $0$ to $n$ can be controlled using the information
from simulating from $-n$ to $n$.
It will be shown shortly that both of these methods enjoy convergence in the local weak sense to the measure $\Psq$.

\begin{lemma}\label{lem:ergodicEFTavg}
    For any strictly increasing sequence of finite intervals $\Fam{I_n}_{n \in \N}$ in $\Z$, and any function
    $g \in L^1(\Psq)$, one has
    \begin{align}
        \frac{1}{\#I_n\E[\#\B_0] } \sum_{v \in V_{I_n}(\B)} g[\overline{\B}, v] \to \Esq[g]
    \end{align}
    and
    \begin{align}
        \frac{1}{\#V_{I_n}(\B)} \sum_{v \in V_{I_n}(\B)} g[\overline{\B}, v] \to \Esq[g],
    \end{align}
    where both convergences happen $\P$-a.s.\ as $n \to \infty$.
    In particular 
    \begin{align}
        \frac{\#V_{I_n}(\B)}{\E[\#V_{I_n}(\B)]}
        =\frac{\#V_{I_n}(\B)}{\#I_n\E[\#\B_0]} \to 1.
    \end{align}
\end{lemma}

\begin{proof}
    Assume without loss that $\Omega = \Xi^\Z$ is the canonical space and $\Fam{\theta_t}_{t \in \Z}$
    is the family of shift operators defined by $\theta_t(\Fam{\xi_s}_{s \in \Z}) = \Fam{\xi_{t+s}}_{s \in \Z}$.
    Both statements follow from rewriting
    \begin{align*}
        \sum_{v \in V_{I_n}(\B)} g[\overline{\B},v] 
        = \sum_{t \in I_n} \left(\sum_{x \in \B_t} g[\overline{\B}, (t,x)]\right)
        = \sum_{t \in I_n} g_0 \circ \theta_t,
    \end{align*}
    where $g_0 := \sum_{x \in \B_0} g[\overline{\B}, (0,x)]$.
    The pointwise ergodic theorem for amenable groups (cf.~\cite{lindenstrauss2001pointwise}) then proves the claim.
\end{proof}

\begin{proposition}\label{prop:localweakconvergence}
    Fix any strictly increasing sequence of finite intervals $\Fam{I_n}_{n \in \N}$ in $\Z$, and
    for each $n \in \N$, let $\bs{o}_n$ be, conditionally on $V_{I_n}(\B)$, uniformly distributed on $V_{I_n}(\B)$ and independent of $\overline{\B} \cap I_n$ (including its marks).
    Then for all bounded measurable $g:\mathcal{G}_*\to \R_{\geq 0}$ depending only on vertices at some bounded distance to the root,
    one has
    \begin{align}
        \frac{1}{\# V_{I_n}(\B)}\sum_{v \in V_{I_n}(\B)}g[\overline{\B}\cap{I_n}, v] \to \Esq[g],\qquad \P\text{-a.s.}
    \end{align}
    as $n \to \infty$.
    In particular,
    \begin{align}
        \P([\overline{\B} \cap I_n, \bs{o}_n] \in \cdot) \to \Psq, \qquad n \to \infty
    \end{align}
    in the sense of local weak convergence.
\end{proposition}

\begin{proof}
    Fix $N\in \N$ and let $g:\mathcal{G}_* \to \R_{\geq 0}$ measurable, bounded, and such that $g$ depends only on vertices at graph distance at most $N$
    from the root.
    One has
    \begin{align*}
        &\E[g[\overline{\B}\cap{I_n}, \bs{o}_n]]\\
        &= \E[\E[g[\overline{\B}\cap{I_n}, \bs{o}_n] \mid V_{I_n}(\B)]] \\
        &= \E\left[\frac{1}{\# V_{I_n}(\B)}\sum_{v \in V_{I_n}(\B)}g[\overline{\B}\cap{I_n}, v]\right] \\
        &= \E\left[\left(\frac{\E[\# V_{I_n}(\B)]}{\# V_{I_n}(\B)}\right) \left(\frac{1}{\E[\#V_{I_n}(\B)]}\sum_{v \in V_{I_n}(\B)}g[\overline{\B}\cap{I_n}, v]\right)\right].
    \end{align*}
    Call the two parenthesized expressions in the previous expectation $a_n$ and $b_n$ respectively,
    then it will be shown that $a_n b_n \to \Esq[g]$ a.s., from which it also follows that $\E[a_n b_n] \to \Esq[g]$ by dominated convergence.
    This will prove the claims.
    By stationarity and linearity of expectation, for each $n \in \N$,
    \begin{align*}
        \Esq[g]
        =\frac{1}{\E[\#\B_{0}]}\E\left[\sum_{v \in \B_{0}}g[\overline{\B}, v]\right]
        =\E\left[\frac{1}{\E[\#V_{I_n}(\B)]}\sum_{v \in V_{I_n}(\B)}g[\overline{\B}, v]\right].
    \end{align*}
    Call the inside of the last expectation $c_n$.
    Letting ${[\Gr, o]}_N$ denote the neighborhood of size $N$ around $o$ in a network $\Gr$,
    for all $n > N$
    \begin{align*}
        &\left\vert b_n - c_n \right\vert \\
        &\leq \frac{1}{\#I_n\E[\#\B_{0}]}\sum_{v \in V_{I_n}(\B)}\left\vert g[\overline{\B}\cap{I_n}, v]-g[\overline{\B}, v]\right\vert\\
        &\leq \frac{2\Vert g\Vert_\infty}{\#I_n\E[\#\B_{0}]}\#\set{v \in V_{I_n}(\B) : {[\overline{\B}\cap{I_n},v]}_N \neq {[\overline{\B},v]}_N}\\
        &\leq \frac{2\Vert g\Vert_\infty}{\#I_n\E[\#\B_{0}]}\left(\sum_{k = \min I_n}^{\min I_n +N} \#\B_{k}  + \sum_{k=\max I_n-N}^{\max I_n} \#\B_{k}\right)\\
        &\leq \frac{2\Vert g\Vert_\infty}{\#I_n\E[\#\B_{0}]}\left(\sum_{k \in I_n} \#\B_{k}  - \sum_{k=\min I_n+ N}^{\max I_n - N} \#\B_{k}\right)\\
        &\to 2\Vert g \Vert_\infty - 2\Vert g\Vert_\infty\\
        &= 0
    \end{align*}
    as $n\to \infty$, $\P$-a.s., by \Cref{lem:ergodicEFTavg}. 
    But also $c_n \to \Esq[g]$ and $a_n \to 1$, $\P$-a.s., also by \Cref{lem:ergodicEFTavg}.
    Hence $a_n b_n \to \Esq[g]$, $\P$-a.s., as claimed.
\end{proof}

\begin{proposition}\label{prop:localweakconvergenceofsimulationeasy}
    Fix any increasing sequence of finite intervals $\Fam{I_n}_{n \in \N} = \Fam{[-a_n, b_n]}_{n \in \N}$ in $\Z$ containing $0$ with $a_n\to \infty$ and $b_n$ strictly increasing.
    For each $n \in \N$, let $\bs{o}'_n$ be, conditionally on $V'_{I_n}(\B)$, uniformly distributed on $V'_{[0,b_n]}(\B)$ and independent of $\overline{\B} \sqcap I_n$ (including its marks).
    Then for all bounded measurable $g:\mathcal{G}_*\to \R_{\geq 0}$ depending only on vertices at some bounded distance to the root,
    one has
    \begin{align}
        \frac{1}{\# (V'_{I_n}(\B) \cap{[0,b_n]})}\sum_{v \in V'_{I_n}(\B) \cap{[0,b_n]}}g[\overline{\B}\sqcap I_n, v] \to \Esq[g],\qquad \P\text{-a.s.}
    \end{align}
    In particular,
    \begin{align}
        \P([\overline{\B} \sqcap I_n, \bs{o}_n'] \in \cdot) \to \Psq, \qquad n \to \infty
    \end{align}
    in the sense of local weak convergence.
\end{proposition}

\begin{proof}
    Fix $N\in \N$ and let $g:\mathcal{G}_* \to \R_{\geq 0}$ measurable, bounded, and such that $g$ depends only on vertices at graph distance at most $N$
    from the root.
    The finiteness of $\B_0$ implies that one has that ${[\overline{\B} \sqcap I_n,v]}_N = {[\overline{\B} \cap I_n, v]}_N = {[\overline{\B}, v]}_N$ eventually as $n \to \infty$ for all $v \in V_{0}(\B)$,
    and hence for all $v \in V_{I_n}(\B) \cap {[0,b_n-N]}$ eventually as $n\to \infty$ as well.
    For the same reason $V'_{I_n}(\B) \cap [0,b_n] = V_{[0,b_n]}(\B)$ eventually as $n \to \infty$ as well.
    It follows that eventually
    \begin{align*}
        &\frac{1}{\#( V'_{I_n}(\B) \cap [0,b_n])}\sum_{v \in V'_{I_n}(\B)\cap[0,b_n]}g[\overline{\B}\sqcap I_n, v]\\
        &=\frac{1}{\#V_{[0,b_n]}(\B)}\sum_{v \in V_{[0,b_n]}(\B)}g[\overline{\B}\cap I_n, v] \\
        &+ \frac{1}{\#V_{[0,b_n]}(\B)}\sum_{v \in V_{[b_n-N+1, b_n]}(\B)}(g[\overline{\B} \sqcap I_n, v] - g[\overline{\B} \cap I_n, v]).
    \end{align*}
    Of the last two terms, $\frac{1}{\#V_{[0,b_n]}(\B)}\sum_{v \in V_{[0,b_n]}(\B)}g[\overline{\B}\cap I_n, v]\to \Esq[g]$
    by \Cref{prop:localweakconvergence}, so it suffices to show that the last term goes to $0$.
    Indeed,
    \begin{align*}
        &\left|\frac{1}{\#V_{[0,b_n]}(\B)}\sum_{v \in V_{[b_n-N+1, b_n]}(\B)}(g[\overline{\B} \sqcap I_n, v] - g[\overline{\B} \cap I_n, v])\right|\\
        &\leq \frac{2\Vert g \Vert_\infty \# V_{[b_n-N+1,b_n]}(\B) }{\#V_{[0,b_n]}(\B)}\\
        &= \frac{2\Vert g \Vert_\infty (\# V_{[0,b_n]}(\B) - \#V_{[0,b_n-N]}(\B)) }{\#V_{[0,b_n]}(\B)}\\
        &\to 2\Vert g \Vert_\infty(1 - 1)\\
        &= 0
    \end{align*}
    as desired.
\end{proof}

\subsubsection{Renewal Structure of the Bridge Graph}\label{sec:renewal-structure-of-mbefts}

In this section, the driving sequence $\xi$ is assumed to be i.i.d., i.e.\  $\G$ is Markovian.
One may ask whether the bridge graph $\B$ admits any kind of renewal structure.
Is it possible that $\B_t$ contains only one state?
This is not necessarily possible.
Indeed, if $p_{x,x} = 0 $, then $\B_t$ contains at least two states for every $t \in \Z$.
It is true, though, that $\B_t$ is infinitely often equal to any set that it has positive probability of being
equal to.
Let $S_{\B}$ denote the \defn{possible configurations} of $\B_0$, i.e.
$S_{\B} := \set{E \subseteq S : \P(\B_0=E)>0}$.
By~\Cref{prop:mbeff-is-locally-finite}, 
$S_{\B}$ consists only of finite subsets of $S$ and is therefore countable.

\begin{lemma}\label{lem:general-regeneration-points}
    For any subset $E \in S_{\B}$,
    the set of $t$ for which $\B_t=E$ forms a simple stationary point process $\Psi_E$ on $\Z$
    with $\P(\Psi_E(\Z)=\infty)=1$ and intensity $\lambda_E =\P(\B_0=E)$.
    In particular, $\Fam{\B_t}_{t \in \Z}$ is bi-recurrent for each $E \in S_\B$.
\end{lemma}

\begin{proof}
    For $E \in S_\B$, the event that there is a $t$ such that $\B_t=E$ is shift-invariant and has positive probability.
    Therefore it happens almost surely.
    The set of such $t$ is shift-covariant and therefore determines a simple stationary point process $\Psi_E$.
    The previous line implies that $\Psi_E$ contains at least one point, and therefore infinitely many a.s.
    One calculates $\lambda_E = \E[\Psi_E(\set{0})] = \E[1_{\set{\B_0=E}}]$, completing the proof.
\end{proof}

Moreover, ruling out obvious hurdles to $\B_t$ being a singleton is sufficient.

\begin{lemma}\label{lem:regeneration-points}
    Suppose $\G$ is an \eft{} and has fully independent transitions.
    Assume that $p_{x^*,x^*}>0$.
    Then $\set{x^*} \in S_{\B}$.
\end{lemma}

\begin{proof}
    By~\Cref{prop:mbeff-is-locally-finite}, $\#\B_t$ is a.s.\ finite for each $t \in \Z$, and thus
    it is possible to choose $x_1,\ldots,x_n \in S$ such that $\P(\B_0 = \set{x_1,\ldots,x_n})>0$.
    Since $\G$ is an \eft{},
    choose a tree $T\subseteq \Z \times S$ with leaves $(0,x_1),\ldots,(0,x_n)$ and root $(t,x^*)$ for some $t>0$
    such that $\P(T \subseteq \G)>0$.
    With $[t] := \set{0,\ldots,t}$, let $I:=\set{s \in [t]: x^* \notin T_s}$.
    Then 
    \begin{align*}
        \P(\B_t = \set{x^*}) 
        &\geq \P(\B_0 = \set{x_1,\ldots,x_n}, T \subseteq \G, F^{(s,x^*)}_{s+1}=x^*, \forall s \in I)\\
        &\geq \P(\B_0 = \set{x_1,\ldots,x_n})\P(T \subseteq \G) \P(F^{(s,x^*)}_{s+1}=x^*, \forall s \in I)\\
        &= \P(\B_0 = \set{x_1,\ldots,x_n})\P(T \subseteq \G) {(p_{x,x})}^{\#I}\\
        &> 0.
    \end{align*}
    To justify the use of independence in the previous calculation, 
    note that $\B_0$ is $\Fam{\xi_{s}}_{s<0}$-measurable,
    whereas the events $\set{T\subseteq \G}$ and $\set{F^{(s,x^*)}_{s+1}=x^*, \forall s \in I}$
    are $\Fam{\xi_{s}}_{s \geq 0}$-measurable, so the first is independent of the second two.
    Then the second is independent of the third because, by construction,
    they involve disjoint sets of edges in $\G$.
\end{proof}

Now it is possible to see the renewal structure in $\B$.
Namely, $\Fam{\B_t}_{t \in \Z}$ is itself an irreducible, aperiodic, and positive recurrent
Markov chain under certain conditions.

\begin{proposition}\label{prop:Bx-renewal-structure}
    One has that $\Fam{\B_t}_{t \in \Z}$ is a Markov chain on $S_{\B}$.
    Additionally, $\Fam{\B_t}_{t \in \Z}$ is stationary and bi-recurrent for every $E \in S_{\B}$.
    Its transition matrix $P_{\B}$ is irreducible and positive recurrent.
    If $\G$ is an $\eft{}$ with fully independent transitions and $p_{x^*,x^*}>0$, then $\set{x^*} \in S_\B$ 
    and $P_{\B}(\set{x^*},\set{x^*})>0$ so $P_{\B}$ is aperiodic as well.
\end{proposition}

\begin{proof}
    By~\Cref{prop:mbeff-is-locally-finite}, $\#\B_t$ is a.s.\ finite for each $t \in \Z$.
    Moreover, $\B_{t+1} = \set{x^*} \cup \set{F^{(t,y)}_{t+1}:y\in \B_t}$, so indeed
    $\Fam{\B_t}_{t \in \Z}$ is a Markov chain on the finite subsets of $S$ since, for each $t \in \Z$,
    $\B_{t+1}$ is a function of $\B_t$
    and $\xi_t$.
    Here the running assumption that $\xi$ is i.i.d.\ is used.
    By~\Cref{lem:general-regeneration-points}, $\Fam{\B_t}_{t\in\Z}$ is bi-recurrent
    for every state $E\subseteq S$ such that $\P(\B_0=E)>0$.
    In particular, the chain must be irreducible on $S_{\B}$, else a return to some state $E_1$ could not occur
    after a return to another state $E_2$ for some $E_1,E_2$ that do not communicate.
    Since $\Fam{\B_t}_{t \in \Z}$ is shift-covariant it is stationary.
    The existence of a positive stationary distribution (the law of $\B_0$) for the irreducible $P_{\B}$ 
    implies $P_{\B}$ is positive recurrent.
    If $\G$ is an \eft{} with fully independent transitions,  then~\Cref{lem:regeneration-points} shows
    that $\set{x^*} \in S_{\B}$.
    Then $p_{x^*,x^*}>0$ implies $P_{\B}(\set{x^*},\set{x^*})>0$ as well,
    so $P_{\B}$ is also aperiodic in that case.
\end{proof}

It is possible that the $S_{\B}$ is strictly smaller than the set of all finite subsets of $S$ containing $x^*$.

\begin{example}
    Consider $S := \set{0,1,2}$ and $x^*:= 0$ with $p_{0,0} = p_{0,1} = p_{0,2} = \frac13$ and $p_{1,0} = p_{2,0} = 1$.
    That is, from $0$ make a uniform choice of where to jump, and from $1$ and $2$ deterministically return to $0$.
    Fix $t \in \Z$.
    In this case, if $1 \in \B_t$, it must be that $F^{(t-1,0)}_t = 1$.
    Similarly, if $2 \in \B_t$, it must be that $F^{(t-1,0)}_t = 2$.
    Thus it cannot be that both $1,2 \in \B_t$, and hence $\set{0,1,2} \notin S_{\B}$.
\end{example}

However, if \emph{every} state has a chance to be lazy, then $S_{\B}$ does turn out to be the set of all finite subsets of $S$ containing $x$.

\begin{proposition}\label{prop:lazy-S-B-is-everything}
    Suppose $\G$ has fully independent transitions, 
    $P$ is irreducible, and  $p_{y,y}>0$ for all $y \in S$.
    Then $\Fam{\B_t}_{t \in \Z}$ is an irreducible, aperiodic, positive recurrent, 
    and stationary Markov chain on the set of all finite subsets of $S$ containing $x^*$.
\end{proposition}

\begin{proof}
    The assumptions imply that, in fact, $P$ is irreducible, aperiodic, 
    and positive recurrent (since $x^*$ is always assumed
    positive recurrent), so \Cref{prop:Bx-renewal-structure} implies that the
    only item left to show is that $S_{\B}$ contains all finite subsets of $S$ containing $x^*$.
    Let a finite set $E$ containing $x^*$ be given.
    Call $(y_1,\ldots,y_n)$ with each $y_i \in S$ a \defn{possible path} if $\prod_{i=1}^{n-1} p_{y_i,y_{i+1}} > 0$.
    For the rest of the proof, all paths considered are possible paths.
    One would like to simply draw a path from $x^*$ to each $y \in S$ where
    after a path reaches its destination it becomes constant while it waits for the other paths to finish.
    This approach is slightly flawed because it may be that,
    for instance, every path from $x^*$ to $z$ passes through $y$.
    In this case, one must draw the path from $x^*$ to $y$ before the path from $x^*$ to $z$, otherwise
    the resulting graph would have a vertex with multiple outgoing edges, which is an impossibility in $\G$.
    However, the approach will work as long as it is possible to draw the 
    paths in an order such that no interference
    occurs.

    \begin{figure}[t!]
    \centering
    \includegraphics[width=\linewidth]{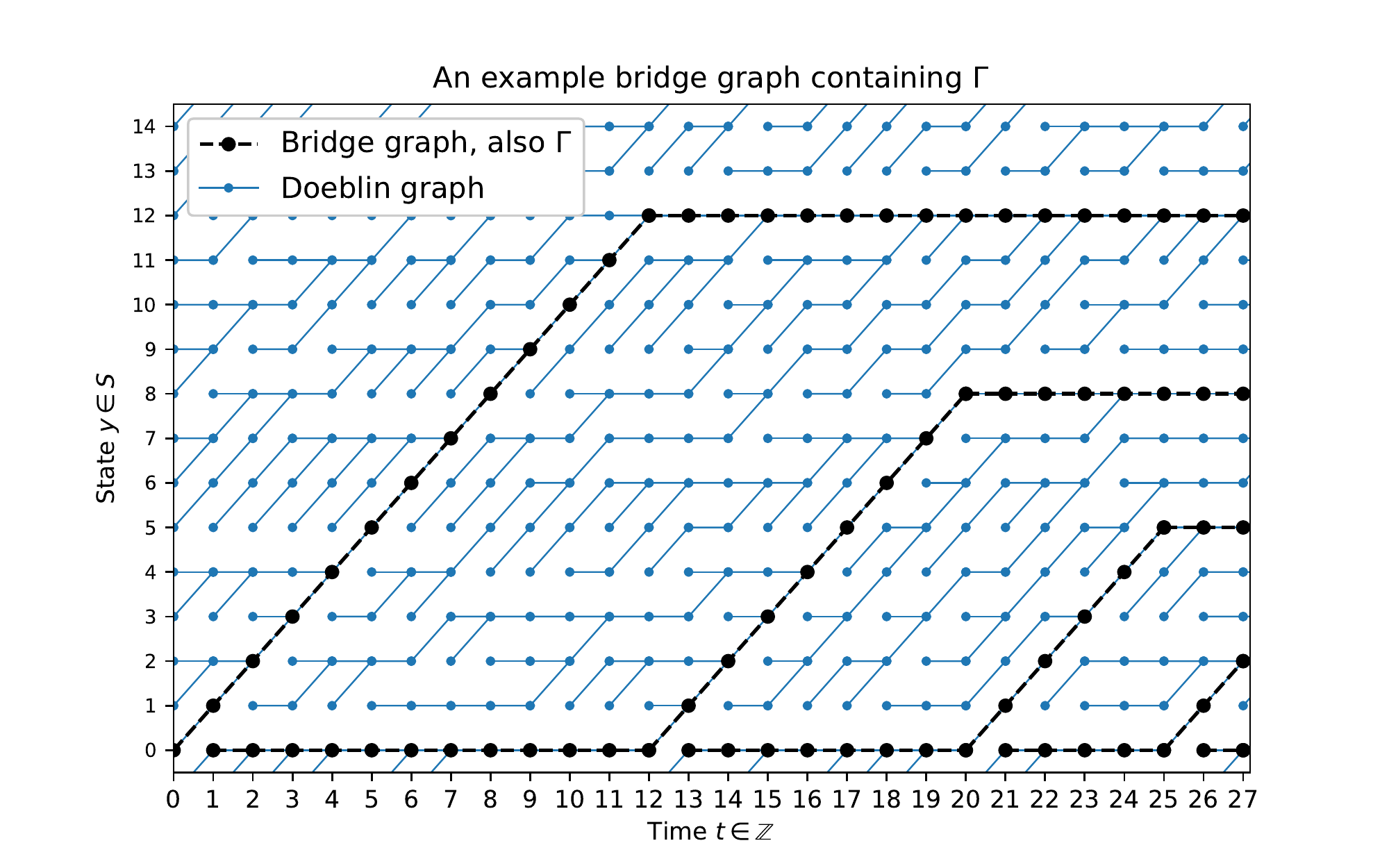}
    \caption{The graph $\Gr$ from the proof of \Cref{prop:lazy-S-B-is-everything}
        when $S=\Z/15\Z$, $x^*=0$, and $E=\set{0,2,5,8,12}$, where
        the Markov dynamics are the lazy version of the deterministic cycle $x \mapsto x+1$ on $S$.
        The graph $\Gr$ is constructed so that, as shown in the figure, 
        if $\B_0 =  \set{0}$ and $\Gr \subseteq \G$,
        then $\B_{27} = \set{0,2,5,8,12}$.}
    \end{figure}

    Define a partial order $\prec_0$ on $E$ by saying $y \prec_0 z$ if all paths from $x^*$ to $z$ pass
    through $y$ with the convention that the trivial path $(x^*)$ does not pass through $x^*$ (to prohibit $x^* \prec_0 x^*$).
    Since $E$ is finite, there is a $\prec_0$-maximal element $x_0 \in E$.
    That is, for all $y \in E$ there is a path from $x^*$ to $y$ that does not hit $x_0$.
    Choose a path $L_0$ from $x^*$ to $x_0$.
    With $\prec_n$, $x_0,\ldots,x_n$, and $L_0,\ldots,L_{n}$ defined,
    as long as $E\setminus \set{x_0,\ldots,x_n} \neq \emptyset$,
    recursively define $\prec_{n+1}$, $x_{n+1}$, and $L_{n+1}$ as follows.
    By construction, for all $y \in E \setminus \set{x_0,\ldots,x_n}$, there is a path from $x^*$ to $y$
    that avoids $x_0,\ldots,x_n$. Define $\prec_{n+1}$ on $E\setminus\set{x_0,\ldots,x_n}$
    by saying $y \prec z$ if all paths from $x^*$ to $z$ avoiding $x_0,\ldots,x_n$ pass through $y$. 
    Then it is possible to choose a $\prec_{n+1}$-maximal element
    $x_{n+1}$, i.e.\ for all $y \in E\setminus\set{x_0,\ldots,x_{n+1}}$,
    there is a path from $x^*$ to $y$ that does not pass through any of $x_0,\ldots,x_{n+1}$.
    Also choose $L_{n+1}$ a path from $x^*$ to $x_{n+1}$ avoiding $x_0,\ldots,x_{n}$.
    Necessarily the recursion terminates when $n=\# E - 1$.
    It is now possible to construct a graph $\Gr\subseteq \Z \times S$
    with $\P(\B_0=\set{x^*}, \Gr \subseteq \G) > 0$ and when $\Gr \subseteq \G$ and 
    $\B_0 = \set{x^*}$,
    one has $\B_t = E$ for some $t$.
    Let $t_i$ be the sum of the lengths of the paths $L_0,\ldots,L_{i-1}$ for each $0 \leq i \leq \#E$, with $t_0:=0$.
    Let $\Gr$ be the graph that for each $i$ has:
    \begin{enumerate}
        \item\label{it:e1} a path from $(t_i,x)$ to $(t_{i+1},x_i)$ with state path $L_i$ from time $t_i$ to $t_{i+1}$,
        \item\label{it:e2} a path started at $(t_{i+1},x_i)$ that stays constant at $x_i$ until time $t_{\#E}$, and
        \item\label{it:e3} a (possibly trivial) path started at $(t_i+1,x)$ that stays constant at $x^*$ until time $t_{i+1}$.
    \end{enumerate}
    Note that, by construction, $\Gr$ is a finite graph that is the union of
	edges that occur with positive probability.
    Moreover, the connected components of $\Gr$ are formed from points \cref*{it:e1,it:e2} for some $i$ and \cref*{it:e3}
    from $i-1$.
    Whenever $\Gr\subseteq \G$ and $\B_0 = \set{x^*}$, one has $\B_{t_{\#E}}= E$.
    This occurs with positive probability since $p_{y,y} > 0$ for all $y \in S$.
\end{proof}

\begin{proposition}\label{prop:PBx-transition-probabilities-recurrence}
    Suppose $\G$ has fully independent transitions.
    Extend the definition of $P_\B$ to 
    \begin{align}
        P_{\B}(E,E') := \P\left(\set{x^*} \cup \set{F^{(0,y)}_1: y \in E} = E'\right),
    \end{align}
    for all finite $E, E' \subseteq S$ containing $x^*$.
    Then $P_{\B}$ satisfies the following recurrence:
    for $E = \set{x^*, x_1,\ldots, x_n}$ and
    $E' = \set{x^*, y_1,\ldots, y_m}$, 
    \begin{align}
        P_{\B}(E,E') 
        &= \left( p_{x_n,x^*} + \sum_{i=1}^m p_{x_n,y_i}\right)P_{\B}(E\setminus \set{x_n}, E')\nonumber\\
        &+\sum_{i=1}^m p_{x_n,y_i}P_{\B}(E\setminus\set{x_n}, E' \setminus
		\set{y_i}),
    \end{align}
    with recursive depth at most $n$ and
    base cases
    \begin{align}
        \begin{cases}
            P_{\B}(E,E') = 0, & \#E' > \#E+1\\
            P_{\B}(\set{x^*}, \set{x^*,y}) = p_{x^*,y},& y \in S\\
            P_{\B}(E, \set{x^*}) = \prod_{y \in E} p_{y,x^*}.
        \end{cases}
    \end{align}
\end{proposition}

\begin{proof}
    First one justifies the extension of the definition of $P_\B$ by noting that
    for $E,E' \in S_{\B}$ one has
    \begin{align*}
        P_{\B}(E,E') &= \P(\B_1 = E' \mid \B_0 = E)\\
        &=\P(\set{x^*} \cup \set{F^{(0,y)}_{1}:y\in \B_0} = E' \mid 
        \B_0 = E)\\
        &=\P(\set{x^*} \cup \set{F^{(0,y)}_{1}:y\in E} = E' \mid 
        \B_0 = E)\\
        &=\P(\set{x^*} \cup \set{F^{(0,y)}_{1}:y\in E} = E' ),
    \end{align*}
    where the last equality follows from the fact that $\B_0$
    is measurable with respect to $\Fam{\xi_{t}}_{t < 0}$, whereas
    $F^{(0,y)}_1$ is $\xi_{0}$-measurable for each $y \in S$.
    The base cases for $P_{\B}$ are immediate from the definition of $P_{\B}$ and the independence structure.
    To see the recurrence, suppose $E = \set{x^*,x_1,\ldots,x_n}$
    and $E' = \set{x^*,y_1,\ldots,y_m}$ as above.
    Split $P_{\B}(E,E')$ depending on the value of $F^{(0,x_n)}_1=x^*$ or $F^{(0,x_n)}_1=y_i$, and on whether 
    $\set{x^*} \cup \set{F^{(0,y)}_1: y \in E\setminus \set{x_n}} = E'$ still or 
    $\set{x^*} \cup\set{F^{(0,y)}_1: y \in E\setminus\set{x_n}} = E'\setminus\set{y_i}$,
    \begin{align*}
        P_{\B}(E,E') 
        &=\P\left(F^{(0,x_n)}_1=x^*,\set{x^*} \cup \set{F^{(0,y)}_1: y \in E\setminus\set{x_n}} = E'\right)\\
        &+\sum_{i=1}^m \P\left(F^{(0,x_n)}_1=y_i,\set{x^*} \cup \set{F^{(0,y)}_1: y \in E\setminus\set{x_n}} = E'\right)\\
        &+\sum_{i=1}^m \P\left(F^{(0,x_n)}_1=y_i,\set{x^*} \cup \set{F^{(0,y)}_1: y \in E\setminus\set{x_n}} = E'\setminus\set{y_i}\right)
    \end{align*}
    which, since $\G$ has fully independent transitions, equals
    \begin{align*}
        &p_{x_n,x^*}\P\left(\set{x^*} \cup \set{F^{(0,y)}_1: y \in E\setminus\set{x_n}} = E'\right)\\
        +&\sum_{i=1}^m p_{x_n,y_i} \P\left(\set{x^*} \cup \set{F^{(0,y)}_1: y \in E\setminus\set{x_n}} = E'\right)\\
        +&\sum_{i=1}^m p_{x_n,y_i}\P\left(\set{x^*} \cup \set{F^{(0,y)}_1: y \in E\setminus\set{x_n}} = E'\setminus\set{y_i}\right)
    \end{align*}
    which simplifies to
    \begin{align*}
        \left(p_{x_n,x^*}+\sum_{i=1}^m p_{x_n,y_i}\right)P_{\B}(E\setminus\set{x_n},E')
        +\sum_{i=1}^m p_{x_n,y_i}P_{\B}(E\setminus\set{x_n}, E'\setminus\set{y_i}),
    \end{align*}
    showing the recurrence holds.

    Finally, the recursive depth needed to fully compute $P_{\B}(E,E')$ is at most $n$
    because each application of the recurrence removes an element from $E$.
\end{proof}

\begin{example}
    By implementing the recurrence of~\Cref{prop:PBx-transition-probabilities-recurrence} in, e.g.\ Python,
    one may compute $P_{\B}$ explicitly.
    Then, given values for the $p_{x,y}$, one may compute the staitonary distribution $\pi_\B$ of $P_\B$.
    For example, with $S:=\set{0,1,2}$ and $x^*:=0$, and $p_{x,y} = \frac{1}{3}$ for all $x,y \in S$,
    one has
    \begin{align*}
        \pi_{\B} 
        &= \begin{bmatrix}
            \pi_{\B}(\set{0}) & \pi_{\B}(\set{0,1}) & \pi_{\B}(\set{0,2}) & \pi_{\B}(\set{0,1,2})
        \end{bmatrix}\\
        &= \begin{bmatrix}
            \frac{17}{143} & \frac{45}{143} & \frac{45}{143} & \frac{36}{143}
        \end{bmatrix}.
    \end{align*}
\end{example}

It is an open question whether, in the fully independent transitions case,
there is a general closed form expression for $P_\B$ in terms of $P$ or for the stationary
distribution $\pi_\B$ of $P_\B$ in terms of $P$ and $\pi$.
\section{Bibliographical Comments}

While this work may be the first time the Doeblin graph $\G$
has been explicitly defined and studied in its own right, it is without doubt that
most, if not all, who have worked on CFTP-related research have had this picture in mind.
Rather, the novelty here lies in the consideration of the bridge graph $\B$.
While, to the best of the authors' knowledge, the bridge graph $\B$ has not previously been defined or studied,
it is not without ties to other objects that have been previously studied.

The first occurrence of some form of the bridge graph appears in~\cite{borovkov1992stochastically},
where Borovkov and Foss consider a family of stochastically recursive sequences started at times $0,-1,-2,\dotsc$,
all with the same initial condition, and they proved the existence (under suitable conditions)
of a stationary version of the SRS\@.
They defined three notions of coupling convergence
and studied when coupling convergence to the stationary SRS occurs.
Their notion of strong coupling convergence to the stationary SRS is akin to the condition that $\B$ is an \eft{}.
That is, it is the condition that all paths in $\B$ eventually merge.
It is conceivable that, in the \eft{} case, one could derive the existence of the bi-infinite path
in $\B$ from the work in~\cite{borovkov1992stochastically}, though it is not clear
whether Borovkov and Foss had this in mind, and they did not make any mention of
the key bi-recurrence property used in the current paper to distinguish
this bi-infinite path from the potential others in $\G$.

Another occurrence of a similar object to the bridge graph may be found in~\cite{baccelli2018renewal}
in the very special case of integer-valued renewal processes.
The dynamics there are slightly different, where instead of specifying a whole process
started from each time, one marks each time with the time of death of an individual who is born at that time.
This is akin to marking each $t \in \Z$ by the return time $\tau^{(t,x^*)}(x^*)$ of $F^{(t,x^*)}$ to $x^*$,
though in~\cite{baccelli2018renewal} these times of death are assumed to be i.i.d.,
whereas in the present work they have intricate dependence due to the Doeblin-type coupling.
The population process defined in~\cite{baccelli2018renewal} is then similar in nature
to the sequence of cardinalities of $\Fam{\B_t}_{t \in \Z}$ as considered in \Cref{sec:renewal-structure-of-mbefts}.
It is proved in~\cite{baccelli2018renewal} that, under natural conditions, the population process is a stationary
regenerative process with independent cycles. In the present work, the process $\Fam{\B_t}_{t \in \Z}$
was shown in \Cref{prop:Bx-renewal-structure} to be an irreducible, aperiodic, and positive recurrent
Markov chain under suitable conditions, which therefore also admits an i.i.d.\ cycle decomposition.
The analysis of this special case and, in particular, the identification of the I/F structure of the
components has been kept in mind throughout the development of the theory of Doeblin \effs{}.
\section{Appendix}
\subsection{Postponed Proofs}

Proofs that were only sketched in the main text are collected in full detail here.

\begin{proof}[Proof of~\Cref{thm:embed-into-general-meff}]
    For each $t \in \Z$, let $\mu_t$ be the distribution of $X_t$.
    It is enough to show the existence of $(\Omega', \F', \P')$
    on which there is a process $X':= \Fam{X'_t}_{t \in \Z}$ and 
    some i.i.d.\ $\xi':= \Fam{\xi'_{t}}_{t\in \Z}$ 
    such that
    \begin{enumerate}
        \item\label{it:c1} $X'_t \sim \mu_t$ for all $t \in \Z$,
        \item\label{it:c2} $\xi_t' \sim \xi_t$ for all $t \in \Z$,
        \item\label{it:c3} $X'_t$ is independent of $\Fam{\xi'_s}_{s \geq t}$ for all $t \in \Z$, and
        \item\label{it:c4}$X'_{t+1} =\hgen(X_t, \xi'_{t})$ for all $t \in \Z$. 
    \end{enumerate}
    \Cref*{it:c1,it:c2,it:c3,it:c4} and \Cref{lem:embed-srs-into-doeblin-graph} will imply the result.
    Note that \cref*{it:c1,it:c2,it:c3,it:c4} are sufficient to characterize the joint finite dimensional distributions
    of $\Fam{X'_t}_{t \in \Z}$ and $\Fam{\xi'_{t}}_{t \in \Z}$.
    To see this fix $t_0 \leq t_1$.
    The joint distribution of
    $\Fam{X'_{t}}_{t_0 \leq t \leq t_1}$ and $\Fam{\xi'_{t}}_{t_0 \leq t \leq t_1}$
    is determined because,
    conditional on $\Fam{\xi'_{t}}_{t_0 \leq t \leq t_1}$,
    $X'_{t_0}$ is still distributed as $\mu_{t_0}$ by \cref*{it:c1,it:c3},
    and, conditional on both $X'_{t_0}$ and $\Fam{\xi'_{t}}_{t_0 \leq t \leq t_1}$,
    one has that $\Fam{X'_{t}}_{t_0 \leq t \leq t_1}$ is deterministic by \cref*{it:c4}.
    Thus it suffices to show that $\Fam{X'_t}_{t \in \Z}$ and $\Fam{\xi'_{t}}_{t \in \Z}$
    satisfying \cref*{it:c1,it:c2,it:c3,it:c4} exist.
    Also note that \cref*{it:c1,it:c2,it:c3,it:c4} with $t_0 \leq t  \leq t_1$ are sufficient for determining the joint distribution of
    $\Fam{X'_s}_{t_0 \leq s \leq t_1}$ and $\Fam{\xi'_{s}}_{t_0 \leq s \leq t_1}$.
    
    The proof will proceed by the Kolmogorov extension theorem.
    Suppose, by extending $(\Omega, \F, \P)$ if necessary, that
    $\Fam{X_t}_{t \in \Z}$ and $\xi$ are defined on the same space and are independent of each other.
    Consider for each $t \in \Z$, the state path $F^{(t,X_t)}$ in $\G$ started at $(t,X_t)$.
    Then for all $s,t \in \Z$ with  $s \leq t$ and all $x \in S$,
    \begin{align*}
        \P(F^{(s,X_s)}_t = x) 
        &= \sum_{y \in S} \P(X_s=y, F^{(s,X_s)}_t=x) \\
        &= \sum_{y \in S} \mu_s(y) P^{t-s}(y,x)\\
        &= \mu_s P^{t-s}(x),
    \end{align*}
    where in the previous line $P$ is treated as a transition kernel with powers $P^k$ ($k=0,1,2,\ldots$).
    Since $\Fam{X_t}_{t \in \Z}$ exists and is a Markov chain with transition matrix $P$,
    one has 
    \begin{align}
        \mu_s P^{t-s} = \mu_{r}P^{s-r}P^{t-s} = \mu_{r} P^{t-r} = \mu_t
    \end{align}
    for all $r \leq s \leq t$.
    Moreover, for all $s\leq t$, $F^{(s,X_s)}_t$ is $\sigma(X_s, \Fam{\xi_{t'}}_{s \leq t' < t})$-measurable,
    hence it is independent of $\Fam{\xi_{t'}}_{t' \geq t}$.
    Now fix $s_0, t_0, t_1 \in \Z$ with $s_0 \leq t_0  \leq t_1$
    and consider the joint distribution of 
    $\Fam{F^{(s_0,X_{s_0})}_t}_{t_0\leq t \leq t_1}$ and  $\Fam{\xi_{t}}_{t_0 \leq t \leq t_1}$.
    One has $F^{(s_0, X_{s_0})}$ and $\xi$ satisfy
    \begin{enumerate}[label=(\roman*')]
        \item\label{it:c1p} $F^{(s_0,X_{s_0})}_t \sim \mu_t$ for all $t_0 \leq t \leq t_1$,
        \item\label{it:c2p} $\xi \sim \xi$,
        \item\label{it:c3p} $F^{(s_0,X_{s_0})}_t$ is independent of $\Fam{\xi_{t'}}_{t' \geq t}$ for all $t_0 \leq t \leq t_1$, and
        \item\label{it:c4p} $F^{(s_0,X_{s_0})}_{t+1} = \hgen(F^{(s_0,X_{s_0})}_t, \xi_{t})$ for all $t_0 \leq t$. 
    \end{enumerate}
    As mentioned before, \cref*{it:c1p,it:c2p,it:c3p,it:c4p} are sufficient to determine the joint distribution of
    $\Fam{F^{(s_0,X_{s_0})}_t}_{t_0 \leq t \leq t_1}$ and $\Fam{\xi_{t}}_{t_0 \leq t \leq t_1}$,
    so the joint distribution of 
    $\Fam{F^{(s_0,X_{s_0})}_t}_{t_0 \leq t \leq t_1}$ and $\Fam{\xi_{t}}_{t_0 \leq t \leq t_1}$
    does not depend on $s_0$ as long as $s_0 \leq t_0$.
    Thus a consistent set of finite dimensional distributions is determined by taking
    $s_0, t_0 \to -\infty$ and $t_1 \to \infty$ while maintaining $s_0 \leq t_0 \leq t_1$.
    It follows by the Kolmogorov extension theorem that there is a space $(\Omega', \F', \P')$ and
    processes
    $X'=\Fam{X'_{t}}_{t \in \Z}$ and $\xi'=\Fam{\xi'_{t}}_{t \in \Z}$ 
    satisfying \cref*{it:c1,it:c2,it:c3,it:c4}, completing the proof.
\end{proof}

Call $P$ \defn{strongly recurrent} if all its recurrent classes are positive recurrent and
call $P$ \defn{recurrent-attracting} if any Markov chain with transition matrix $P$ eventually
enters a recurrent state.
These conditions are both automatic if $P$ is irreducible and positive recurrent.

\begin{proposition}[Subsumes~\Cref{prop:teaser-indep-transitions-components}]\label{prop:indep-transitions-components}
    Let $S = T \cup \bigcup\Fam{R^i}_{0 \leq i< N}$ decompose $S$ into its transient states and
    $N\in\N\cup\set{\infty}$ recurrent communication classes for $P$.
    Assume that $P$ is strongly recurrent and recurrent-attracting.
    Let $d(i)$ be the period of $R^i$, and let $R^i = C^i_0 \cup \cdots \cup C^i_{d(i)-1}$
    be a cyclic decomposition.
    If $\G$ has fully independent transitions, then the components 
    $\Fam{\mathcal{C}^i_j}_{0 \leq i < N, 0 \leq j < d(i)}$ of $\G$ are in bijection
    with $\Fam{C^i_j}_{0 \leq i < N, 0 \leq j < d(i)}$, and for $x \in C^i_j$,
    $(t,x) \in V(\mathcal{C}^i_{j'})$ if and only if $j-t=j' \pmod{d(i)}$, and for $x \in T$, $(t,x) \in V(\mathcal{C}^{i}_{j})$
    where $(t',y) \in V(\mathcal{C}^{i}_{j})$ is any vertex on the path of $(t,x)$ for which $y$ is recurrent.
    That is,  $\mathcal{C}^i_{j}$ is the set of all vertices of all paths in $\G$
    that pass through an element of $C^i_j$ at any time $t = 0 \pmod{d(i)}$.
\end{proposition}

\begin{proof}
    Fix $i,j,t$ and let $x,y \in C^i_j$.
    Then $P^{d(i)}$ restricted to $C^i_j$ is irreducible, aperiodic, and positive recurrent.
    Thus the product chain $P^{d(i)} \otimes P^{d(i)}$ restricted to $C^i_j \times C^i_j$ is too.
    Strictly before the hitting time to the diagonal, $\Fam{F^{(t,x)}_{t+ sd(i)}, F^{(t,y)}_{t+sd(i)}}_{s \geq 0}$ is distributed
    the same as the product chain $P^{d(i)} \otimes P^{d(i)}$ on $C^i_j \times C^i_j$,
    and thus the hitting time to the diagonal is a.s.\ finite because the product chain is
    irreducible, aperiodic, and positive recurrent.
    It follows that $(t,x)$ and of $(t,y)$ are in the same component of $\G$.
    If $x \in C^i_j$ and $y \in C^{i'}_{j'}$ with $i' \neq i$, then $F^{(t,x)}$ and $F^{(t,y)}$ cannot merge because
    the states of $F^{(t,x)}$ are contained in $R^i$ and the states of $F^{(t,y)}$ is contained in $R^{i'}$.
    If $x \in C^i_j$ and $y \in C^i_{j'}$ with $j' \neq j \pmod{d(i)}$, then $F^{(t,x)}$ and $F^{(t,y)}$ cannot merge
    because $F^{(t,x)}_{t+s} \in C^i_{j+s}$ but $F^{(t,y)}_{t+s} \in C^i_{j'+s}$ with indices taken modulo $d(i)$
    as necessary.
    Thus, the set of $y \in S\setminus T$ such that $F^{(t,x)}$ eventually merges with $F^{(t,y)}$ 
    is precisely $C^i_j$.
    If $x \in C^i_j, y \in C^{i'}_{j'}$ and $t \leq t'$, then
    $F^{(t,x)}_{t'} \in C^i_{j+(t'-t)}$, so it follows
    that $F^{(t,x)}$ and $F^{(t',y)}$ eventually merge if and only if $i'=i$ and $j+(t'-t) = j' \pmod{d(i)}$, or equivalently
    $j'-t'= j-t \pmod{d(i)}$.
    It follows that for any $x,y \in S\setminus T$ and any $t,t' \in \Z$, the two vertices
    $(t,x),(t',y) \in V(\G)$ are in the same component of $\G$ if and only if there are $i,j,j'$
    such that $x \in C^i_j, y \in C^i_{j'}$ and $j'-t' = j-t \pmod{d(i)}$.
    If $x \in C^i_j$, then $F^{(t,x)}_{s} \in C^i_{j-t+s}$ for all $s \geq t$.
    Thus $F^{(t,x)}_{s} \in C^i_{j-t}$ for all $s = 0 \pmod{d(i)}$ with $s\geq t$.
    Call $C^i_{j-t}$ the \defn{time-zero class of $(t,x)$}.
    Then for $x\in C^i_j, y\in C^i_{j'}$ and $t,t' \in \Z$, the condition that
    $j'-t' = j-t \pmod{d(i)}$ is equivalent to the fact that $(t,x)$ and $(t,y)$ have the same time-zero
    class.
    Thus the components of $\G$ are exactly the equivalence classes of vertices in the same time-zero class,
    except possibly ignoring $(t,x)$ for transient $x$.
    By the assumption that $P$ is recurrent-attracting, if $x \in T$, then $F^{(t,x)}$ eventually hits some recurrent 
    class and so does not form a new component of $\G$, and the path $F^{(t,x)}$ is in the component
    of the first (and every) $(t',y)$ it hits with $y$ recurrent.
\end{proof}

\begin{proof}[Proof of~\Cref{lem:class-of-tangible-network-is-measurable}]
    Fix $k \in \N$.
    For every $v \in V=\Z\times S$,
    the event that $v \in V(\bs{\Gr})$ and $d_{\bs{\Gr}}(\bs{o},v) \leq k$ is measurable.
    Indeed, there are at most countably many paths $(v_0,v_1,v_2,\ldots,v_n)$ in $V$ with $n\leq k$, and
    the desired event is the union over all such paths of any length $n \leq k$ ending at $v$ of the event
    \begin{align*}
        \set{\bs{o}=v_0} \cap \bigcap_{i=1}^n \left(\set{f_V(v_i)=1} \cap \set{f_E(v_{i-1},v_{i})=1}\right).
    \end{align*}
    From here one sees that event that the $r$-neighborhood around $\bs{o}$ is exactly some fixed finite graph $\Gr$
    is measurable.
    Indeed,
    \[
        \set{N_{\bs{\Gr}}(\bs{o},r) = \Gr} 
        = \bigcap_{v \in V} \set{(f_V(v)=1 \text{ and } d_{\bs{\Gr}}(\bs{o},v) \leq r) \iff v \in V(\Gr)}.
    \]
    Enhancing $\Gr$ with marks $\xi_u, \xi_{v,w}$ for each $u \in V(\Gr)$ and all $\set{v,w} \in E(\Gr)$,
    for any $\epsilon > 0$ and $o \in V(\Gr)$, one sees that the event
    \begin{align*}
        D_{r,\epsilon}(\Gr,o) := &\{
        \bs{o}=o,\\
        &N_{\bs{\Gr}}(\bs{o},r) = \Gr,\\
        &\forall u \in V(\Gr), d_{\Xiuniv}(\xi_V(u), \xi_u) < \epsilon,\\
        &\forall  \set{v,w} \in E(\Gr), d_{\Xiuniv}(\xi_E(v,w), \xi_{v,w})<\epsilon\}
    \end{align*}
    is measurable.
    Since $V$ is countable and $\Gr$ is a finite graph, there are at most
    countably many rooted isomorphic copies of $(\Gr, o)$ that can be made with
    vertices in $V$.
    It follows that the event $\set{d_{\mathcal{G}_*}([\bs{\Gr},\bs{o}], [\Gr,o]) < \epsilon}$ is a countable
    union of the events $D_{\lceil \frac{1}{\epsilon} \rceil,\epsilon}(\rho(\Gr,o))$ with $\rho$ 
    ranging over the countable collection of such rooted isomorphisms of $(\Gr, o)$.
    Hence $\omega \mapsto [\bs{\Gr}(\omega),\bs{o}(\omega)]$ is measurable.
\end{proof}

\begin{proof}[Proof of~\Cref{prop:naiive-unimodularity-in-G}]
	First suppose that $X_0$ is independent of $\G$ and uniformly distributed
    on a finite $S$.
	Let $N$ be the cardinality of $S$.
	Let $g:\mathcal{G}_{**} \to \R_{\geq 0}$ supported on directed neighbors be
	given. Then
	\begin{align*}
		\E\sum_{v \in V(\G)} g[\G,(0,X_0), v]
		&= \frac{1}{N}\sum_{x \in S}\E\sum_{y\in S} g[\G,(0,x),(1,y)]\\
		&= \frac{1}{N}\sum_{x,y\in S}\E [g[\G,(0,x),(1,y)]]\\
		&= \frac{1}{N}\sum_{x,y\in S}\E [g[\G,(-1,x),(0,y)]]\\
		&= \frac{1}{N}\sum_{y \in S}\E \sum_{x \in S} g[\G,(-1,x),(0,y)]\\
        &= \E\sum_{v \in V(\G)} g[\G,v,(0,X_0)],
	\end{align*}
	where in the third equality time-homogeneity of $\G$ is used.
	It follows that in this case $\G$ is unimodular.
    
    Next suppose $[\G, (0,X_0)]$ is unimodular.
    Let $\eta$ be a vertex-shift that follows the arrows in $\G$.
    For example, define for each network $\Gr$ and $u \in V(\Gr)$
    the vertex-shift by $\eta_\Gr(u):= v$ if there is a unique outgoing edge from $u$ and this edge terminates at $v$,
    or $\eta_\Gr(u):= u$ if this condition is not met for any $v$.
    Since $\G$ is connected, its $\eta$-foils are $\Fam{\G_t}_{t \in \Z}$.
    Let the mark of a vertex $v$ be denoted $(s(v), \xi(v))$, and let $v \sim w$ denote
    that $v$ and $w$ are in the same $\eta$-foil.
    Fix $x,y \in S$ and let $g[G,v,w] := 1_{\set{s(v) = x, s(w)=y, v \sim w}}$.
    Then the mass transport principle implies
    \[
        \P( X_0 = x) =\E \sum_{v \in V(\G)} g[\G, (0,X_0), v] = \E\sum_{v \in V(\G)} g[\G, v, (0,X_0)] = \P(X_0 = y),
    \]
    so $X_0$ is uniformly distributed on $S$.

    Next let $X_0$ be the output of the CFTP algorithm in the standard CFTP setup.
    Suppose $[\G, (0,X_0)]$ is unimodular.
    Since $(0,X_0)$ has one outgoing edge in $\G$, unimodularity implies that on average it has
    one incoming edge.
    But, being the output of the CFTP algorithm, $(0,X_0)$ a.s.\ has at least one incoming edge.
    Hence $(0,X_0)$ a.s.\ has exactly one incoming edge.
    By unimodularity, it follows that a.s.\ every vertex in $\G$ has exactly one incoming edge.
    Since $\G$ is a tree, this is only possibly if $S$ has a single element.
    If $S$ has only a single element unimodularity is immediate.
\end{proof}

\subsection{List of Mass Transports}\label{sec:list-of-mass-transports}
As mentioned in \Cref{sec:MBEFFs}, the proof style of \Cref{prop:mbeff-is-locally-finite}
can be used to prove many equalities and inequalities in mean.
A list is provided giving mass transports, followed by the results they give after applying the boilerplate
proof style with these mass transports.
Drawing a picture for each transport helps significantly in computing $w^+$ and $w^-$ for the given transports.
In all of the following, 
$\beta$ is the union of all bi-recurrent paths in $\B$.

\begin{enumerate}
    \item Send mass $1$ from each $s$ to all times $t$ strictly after $s$ and strictly before $F^{(s,x^*)}$ returns to $x^*$.
    \item[$\bullet$] $\E[\#\B_0] \leq \E[\sigma^{(0,x^*)}(x^*)]$, where
	    $\sigma^{(0,x^*)}(x^*)$ is the time until return of $F^{(0,x^*)}$ to $x^*$.
    \item Fix $y \in S$. For each $s$, if $y \in \B_s$, send mass $1$ to the first time $t>s$ that $F^{(s,y)}$ hits $x^*$.
    \item[$\bullet$] $\P(y \in \B_0) = \E[\#\TR(0)^y]$, where $\TR(0) \subseteq\B$ is the subgraph of vertices that first return to $x^*$ at time $0$,
    i.e., the (possibly empty) subgraph of $\B$ of all $(t,y) \in V(\B)$ such that $\tau^{(t,y)}(x^*) = 0$,
    where $\tau^{(t,y)}(x^*)$ is the return time of $F^{(t,y)}$ to $x^*$.
    \item[$\bullet$] Summing over $y \in S$, one finds $\E[\#V(\TR(0))] = \E[\#\B_0]$.

    \item\label{it:f3} Send mass $1$ from each $s$ to the first time $t>s$ that $F^{(s,x^*)}_t = F^{(s', x^*)}_t$ for some $s'>t$.
    \item[$\bullet$] $\E[C(0)] = 1$, where $C(0)$ is the total number of paths $F^{(s,x^*)}$ that merge with a younger $F^{(s',x^*)}$ (i.e.\ with $s'>s$)
        for the first time at time $0$.
    \item[$\bullet$] $\P(\#\B_1 \leq \#\B_0-k) \leq \P(C(1) \geq k+1) \leq \frac{1}{k+1}$ for all $k \in \N$.
    \item 
        Fix $y \in S$.
        For each $s$, send mass $1$ to each time $t$ that $F^{(s,x^*)}_t = y$ and $t$ is strictly before $F^{(s,x^*)}$ merges with the unique bi-recurrent path in its component of $\B$.
    \item[$\bullet$]
        $\E[ N^{(0,x^*)}_0(y ; \beta)] = \E[1_{\set{y \in \B_0\setminus\beta_0}}\#V^{x^*}(D^{(0,y)} \cap V(\B))]$,
        where $N^{(0,x^*)}_0(y ; \beta)$
        denotes the number of visits (potentially $0$) of $F^{(0,x^*)}$ to $y$ strictly before merging with $\beta$.
    \item[$\bullet$]
        Summing over $y \in S$,
        $\E[ \sigma^{(0,x^*)}_0(\beta)] = \E[\#V^{x^*}(D^{V_0(\B)\setminus V_0(\beta)} \cap V(\B))]$,
        where $\sigma^{(0,x^*)}_0(\beta)$ is the number of steps (potentially $0$) before $F^{(0,x^*)}$ merges with $\beta$,
        and $D^{V_0(\B)\setminus V_0(\beta)}$ is the set of all descendants of all $v \in V_0(\B)\setminus V_0(\beta)$.
    \item Fix $y \in S$. For each $s$, if $y \in \B_s$ send mass $1$ to the first time $t$ that $F^{(s,y)}$
        is on the bi-recurrent path in its component of $\B$.
    \item[$\bullet$] $\P(y \in \B_0) = \E[\# V^{y}(D^{V_0(\beta), M} \cap V(\B))]$, where $D^{V_0(\beta), M}$ denotes the union of $V_0(\beta)$
        with their \defn{mortal descendants}, i.e.\ those descendants with only finitely many descendants and whose first ancestor in $\beta$ is at time $0$.
    \item[$\bullet$] Summing over $y \in S$, one finds $\E[\#\B_0] = \E[\# (D^{V_0(\beta),M} \cap V(\B))]$.
    \item Fix $y \in S$ and suppose $\G$ is an \eft{} and $\Fam{\beta_t}_{t \in \Z}$ is the bi-recurrent path in $\G$.
        For each $t$, if $\beta_t = y$ send mass $1$ backwards to the most recent time $s<t$ such that $\beta_s = x^*$.
    \item[$\bullet$] $\E[N^{(0,x^*)}(y;x^*) 1_{\beta_0 = x^*}] = \P(\beta_0 = y)$,
        where $N^{(0,x^*)}(y;x^*)$ denotes the number of visits of $F^{(0,x^*)}$ to $y$
        before returning to $x^*$, including the initial visit if $y=x^*$.
    \item[$\bullet$] Summing over $y \in S$, one finds $\E[\sigma^{(0,x^*)}(x^*) 1_{\set{\beta_0 = x^*}}] = 1$.
    \item[$\bullet$] If $\G$ is also Markovian, then the previous points reduce to the classical cycle formulas,
        $\E[N^{(0,x^*)}(y;x^*)] \pi(x^*) = \pi(y)$ and $\E[\sigma^{(0,x^*)}(x^*)] \pi(x) = 1$,
        where $\pi$ is the stationary distribution of the Markov chain.
\end{enumerate}

Instead of using the unimodularity of $\Z$ and specifying a mass transport $w=w(s,t)$ for $s,t \in \Z$, 
one may also use the unimodular version of $\B$ (that is, the random network with distribution $\Psq$) and specify a 
mass transport $w=w[\Gr, u,v]$ for all networks $\Gr$ and all $u,v \in V(\Gr)$.
Some mass transports are much easier to write in this way.
For example, the mass transport in \cref*{it:f3} above also follows from the mass transport $w[\Gr,u,v] = 1$ if $v$ is
the unique out-neighbor of $u$ in $\Gr$.
However, strictly speaking, there are no results using a mass transport on $\B$
that could not also be proved with a mass transport on $\Z$.
Indeed, if $w$ is a mass transport defined for all networks $\Gr$,
then with $[\overline{\B},\sq]$ denoting the identity map under $\Psq$,
\begin{align*}
    \Esq\left[\sum_{v \in V(\B)} w[\overline{\B},\sq, v]\right] = \Esq\left[\sum_{v \in V(\B)} w[\overline{\B}, v, \sq]\right]
\end{align*}
may be rewritten as
\begin{align*}
    \frac{1}{\E[\#\B_0]}\E\left[\sum_{t \in \Z}\hat{w}(0,t)\right]
    = \frac{1}{\E[\#\B_0]}\E\left[\sum_{t \in \Z}\hat{w}(t,0)\right]
\end{align*}
where
\begin{align*}
    \hat{w}(s,t) := \sum_{u \in V_s(\B)} \sum_{v \in V_t(\B)} w[\overline{\B},u,v],\qquad s,t \in \Z
\end{align*}
is a mass transport on $\Z$.
That being said, the reader is encouraged the ponder
the sequence of mass transports on $\Z$ that would be required to prove
a result like the classification theorem, \Cref{thm:foil-classification-theorem},
for the network $[\overline{\B},\sq]$ directly.
It seems more elegant to call upon the machinery of unimodular networks
when convenient instead.

\section*{Acknowledgments}
This work was supported by a grant of the Simons Foundation (\#197982 to The
University of Texas at Austin).
The second author thanks the Research and Technology Vice-presidency of Sharif
University of Technology for its support.

\begingroup
    \section{References}
    \renewcommand{\section}[2]{}
    \bibliographystyle{abbrv}

\endgroup
\end{document}